\documentclass[review,onefignum,onetabnum]{siamart171218}

\usepackage{lipsum}
\usepackage{amsfonts}
\usepackage{graphicx}
\usepackage{subfigure}
\usepackage{epstopdf}
\usepackage{algpseudocode}
\usepackage{bm}

\nolinenumbers
\ifpdf
  \DeclareGraphicsExtensions{.eps,.pdf,.png,.jpg}
\else
  \DeclareGraphicsExtensions{.eps}
\fi

\newsiamremark{remark}{Remark}
\newsiamremark{hypothesis}{Hypothesis}
\crefname{hypothesis}{Hypothesis}{Hypotheses}
\newsiamthm{claim}{Claim}

\headers{Decoupling PDE Computation with Intrinsic {or Inertial} Robin Interface Condition}{Mo Mu and Lian Zhang}

\title{Decoupling PDE Computation with Intrinsic {or Inertial} Robin Interface Condition\thanks{Submitted to the editors DATE.
\funding{This work was funded by the Hong Kong RGC Competitive Earmarked Research Grant HKUST16301218 and NSFC (91530319,11772281).}}}

\author{
Mo Mu \thanks{Department of Mathematics, Hong Kong University of Science and Technology
	(\email{mamu@ust.hk}).}
\and 
Lian Zhang\thanks{Department of Mathematics, Hong Kong University of Science and Technology
	(\email{lzhangay@ust.hk}).}
}

\usepackage{amsopn}

\begin{document}
	
	\maketitle
	
	\begin{abstract}
	We study decoupled numerical methods for multi-domain, multi-physics applications.
	By investigating various stages of numerical approximation and decoupling and tracking how the information is transmitted across the interface for a typical multi-modeling model problem, we derive an approximate intrinsic or inertial type Robin condition for its semi-discrete model.
	This new interface condition is justified both mathematically and physically in contrast to the classical Robin interface condition conventionally introduced for decoupling multi-modeling problems. Based on the intrinsic or inertial Robin condition, an equivalent semi-discrete model is introduced, which provides a general framework for devising effective decoupled numerical methods. Numerical experiments also confirm the effectiveness of this new decoupling approach.
	\end{abstract}
	
	
	\section{Introduction}
	Multi-domain, multi-modeling problems find important applications in science and engineering. They usually involve coupled processes or systems in more than one geometric domains or physical fields, and arise from classical parallel computing, domain decomposition, to contemporary multi-scale and multi-physics computation.
	Mathematically, they are modeled by different partial differential equations(PDEs) in different domains with certain interface coupling conditions, such as fluid-structure interactions (FSI) \cite{crosetto2011fluid,leng2019numerical,turek2006proposal}, coupling fluid flow with porous media flow \cite{badea2010numerical,cao2010coupled,discacciati2009navier}, etc..
	
	Numerical methods for multi-domain, multi-modeling problems are generally classified as monolithic and partitioned approaches.
	The monolithic approach usually leads to coupled schemes \cite{barker2010scalable,riviere2005analysis,riviere2005locally,wick2013solving}. Although unconditionally stable and convergent in general, those coupled schemes usually result in significant difficulties and inflexibility in the design and choice of mesh generation, PDE discretization, algebraic solvers, as well as software development.
	
	On the other hand, certain partitioned approaches, often called loosely coupled, explicit coupling, or decoupled schemes, have also been investigated in the literature \cite{badia2008fluid,bukac2016stability,cai2009numerical,canic2012stability,discacciati2007robin,fernandez2013explicit,fernandez2015generalized,mu2007two,mu2010decoupled}, where the sub-model in each domain is solved locally so that existing legacy solvers may be applied by easy software integration. There are many other important physical and numerical considerations that appeal to partitioned approaches for treating different sub-models in their own physical regions independently, particularly because of the very different physical properties and time scales.
	However, decoupled approaches are usually theoretically much more difficult. Decoupling could easily lead to numerical instability or non-convergence. For instance, many decoupled methods in FSI applications are so-called unconditionally instable due to the artificial added-mass effect \cite{causin2005added,forster2007artificial}. Even with classical non-overlapping domain decomposition,
	for a long time people have been still searching for and trying to understand how to decouple the computation effectively to ensure stability and convergence.
	The most typical and conventional approach is to match the given interface conditions, for instance the continuity of the Dirichlet and Neumann data, in one way or another by
	solving local sub-models with certain given boundary conditions and then exchanging the information across the interface in certain ad hoc or intuitive way, which often leads to
	decoupled methods of the Dirichlet-Dirichlet or Dirichlet-Neumann type. These methods are sometimes stable and convergent, sometime not, or not very efficient.
	A Robin type condition has been introduced by combining the original Dirichlet and Neumann interface conditions to obtain a set of mathematically equivalent interface conditions, which results in some other types of decoupled methods such as Dirichlet-Robin, Neumann-Robin, and Robin-Robin type of schemes. Sometimes, this may help to improve the effectiveness and efficiency.
	However, such an approach is artificial and lack of physical justification, which results in both theoretical and practical numerical difficulties for the Robin type decoupling approach.
	For instance, in applications like FSI the Dirichlet data correspond to velocity, while the Neumann data correspond to interfacial force, therefore, there is a mismatch between these two different types of physical quantities. In the classical Robin approach, this physical mismatch is somehow meredied by mathematically introducing a relaxation parameter, which is to be tuned for the purposes of numerical stability and convergence. The optimal relaxation parameter is then compensated for the mathematical effect of the physical mismatch. However, the optimal parameter analysis is often very difficult, even for certain simple model problems. The choice and tuning of the relaxation parameter with the existing numerical methods are practically very difficult and even impossible for most of the real applications, never mentioning the theoretical analysis and understanding of the numerical stability and convergence.
	
	In an interesting FSI application of a thin-wall structure coupled with a fluid flow, a so-call intrinsic Robin condition was introduced with certain formal derivation, which leads to a stable decoupled scheme and overcomes the added-mass instability effect \cite{fernandez2013explicit,fernandez2015generalized}. Motivated by this, we believe that this could provide certain insight
	and help understand the interface mechanism, as well as provide guidance for developing decoupling methodology for general multi-domain and multi-physics applications.
	
	We investigate a typical multi-modeling model problem. By cruising various stages of numerical approximation and decoupling, and track how the information is transmitted across the interface, we clearly demonstrate how the interface conditions evolve and affect the data and error transmission across the interface during spacial discretization, time discretization, as well as decoupling.
	An approximate intrinsic or inertial type Robin condition is rigorously derived for a semi-discrete model with finite element spacial discretization,
	which makes sense both mathematically and physically in contrast to the classical Robin condition.
	Based on this new interface condition, an equivalent semi-discrete model is introduced, which provides a general framework for devising effective decoupled numerical methods.
	Numerical experiments also confirm the effectiveness of this new decoupling approach.

	The paper is organized as follows. The multi-domain model problem is described in Section 2. The derivation and analysis of the new intrinsic or inertial type Robin interface condition are carried out in Section 3. The decoupling approach based on the intrinsic or inertial type Robin interface condition is introduced in Section 4, followed by numerical experiments in Section 5. Concluding remarks are given in Section 6.
	
	\section{Model Problem}
	We consider a scalar coupled model consisting of two heat equations in two adjacent subdomains coupled through interface conditions.
	Let $\Omega_1=(0,1)\times(0,1)$ and $\Omega_2=(1,2)\times(0,1)$ be the two subdomains, with the interface $\Gamma=\partial\Omega_1\cap\partial\Omega_2$ as showed in Figure \ref{ComputationalDomains}, where
	${\bm n}_1$ and ${\bm n}_2$ are the unit outward normal vectors to $\partial\Omega_1$ and $\partial\Omega_2$, respectively.
	
	\begin{figure}[htbp!]
		\centering
		\includegraphics[scale=0.40]{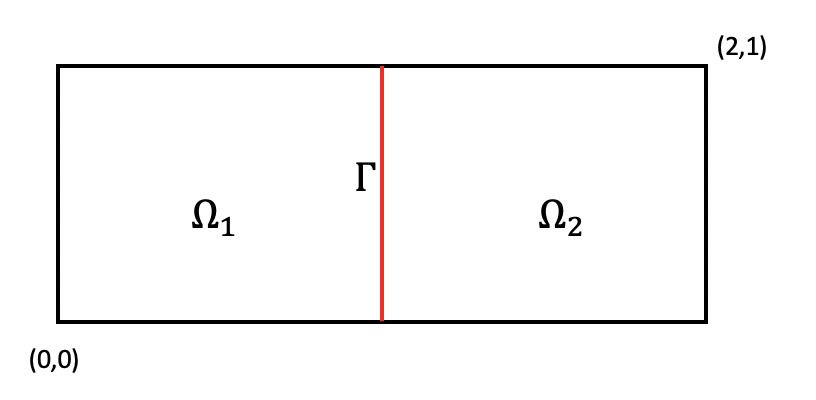}
		\caption{Computational Domains with Interface}
		\label{ComputationalDomains}
	\end{figure}
	
	The coupled PDE model is described as:
	\begin{eqnarray}
	\rho_1\frac{\partial u_1}{\partial t} - \nabla \cdot (\beta_1 \nabla u_1) = f_1, &\mbox{~~in~~}& \Omega_1,\label{Equation1}\\
	\rho_2\frac{\partial u_2}{\partial t} - \nabla \cdot (\beta_2 \nabla u_2) = f_2, &\mbox{~~in~~}& \Omega_2,\label{Equation2}\\
	u_1  = u_2, &	\mbox{~~on~~}& \Gamma,\label{InterfaceCondition1}\\
	\beta_1 \nabla u_1 \cdot {\bm n}_1 + \beta_2 \nabla u_2 \cdot {\bm n}_2=q, &\mbox{~~on~~}& \Gamma\label{InterfaceCondition2},\\
	u_1 = u_{1D}, &\mbox{~~on~~}& \partial\Omega_1\backslash\Gamma\label{BoundaryCondition1},\\
	u_2 = u_{2D}, &\mbox{~~on~~}& \partial\Omega_2\backslash\Gamma\label{BoundaryCondition2},
	\end{eqnarray}
	where $u_1$ and $u_2$ are temperature, $\rho_1$ and $\rho_2$ are density, $\beta_1$ and $\beta_2$ are material diffusivity, $f_1$ and $f_2$ are external source, and  the Dirichlet and Neumann conditions are imposed on the interface by (\ref{InterfaceCondition1}) and (\ref{InterfaceCondition2}).
	For illustration and without loss of generality, we will assume $q=0$.
	
	It serves as a typical model for investigating decoupling techniques for various important applications, bridging classical domain decomposition and parallel computing with contemporary multi-physics computing.
	When $\rho_1 = \rho_2$ and $\beta_1 = \beta_2$, this coupled model is equivalent to a underline single model of a heat equation defined over the global domain, being partitioned to subdomain problems with non-overlapping domain decomposition so that various domain decomposition time marching schemes or iterative methods may be devised for parabolic or elliptic problems with parallel computing.
	When $\beta_1 \neq \beta_2$, and/or $\rho_1 \neq \rho_2$, it may be used as a model problem to study decoupling techniques for jumping coefficients and multi-scale applications.
	The scalar model problem may also be used to investigate and understand the numerical and physical properties of decoupling techniques for being further extended to other important applications in multi-modeling, multi-physics, multi-scale computation such as FSI.

Define the function spaces

\begin{equation}\label{SpaceDefinition}
\begin{aligned}
&{\bm V}_1 = \left\{ v_1 \in {\bm H}^1 ({\Omega}_1):~ v_1=0 \mbox{~~on~~} \partial\Omega_1\backslash\Gamma \right\},\\
&{\bm V}_2 = \left\{ v_2 \in {\bm H}^1 ({\Omega}_2):~ v_2=0 \mbox{~~on~~} \partial\Omega_2\backslash\Gamma \right\},\\
&{\bm V} = \left\{ (v_1,v_2) \in {\bm V}_1 \times {\bm V}_2 :~
v_1 = v_2 \mbox{~~on~~} \Gamma \right\},\\
&{\tilde{\bm V}} = \left\{ (v_{1},v_{2}) \in {\bm V}_{1} \times {\bm V}_{2}\right\},\\
&{\bm V}^{\Gamma} = \left\{ (v_1,v_2) \in {\bm V}_1 \times {\bm V}_2:~
v_1 = v_2 =0 \mbox{~~on~~} \Gamma \right\},\\
&{\bm \Lambda}_{\Gamma} = \{\xi = v_1|_{\Gamma}:~ v_1\in {\bm V}_1\}=\{\xi = v_2|_{\Gamma}:~ v_2\in {\bm V}_2\},\\
&{\bm V}_{1}^{\Gamma} = \{v_1\in {\bm V}_1:~ v_1=0 \mbox{~~on~~} \Gamma\},\\
&{\bm V}_{2}^{\Gamma} = \{v_2\in {\bm V}_2:~ v_2=0 \mbox{~~on~~} \Gamma\}.
\end{aligned}
\end{equation}

Integration by parts gives the weak form of the coupled model (\ref{Equation1})-(\ref{BoundaryCondition2}):

{\bf Model C: The Weak Form of the Continuous Coupled Model}\\

Find $u_{}^{C}=(u_{1}^{C},u_{2}^{C})\in {\bm V}$, such that
\begin{equation}\label{ModelC}
\begin{split}
\rho_1\left(\frac{\partial u_{1}^{C}}{\partial t},v_{1}\right)_{\Omega_1} + \left(\beta_1 \nabla u_{1}^{C}, \nabla v_{1} \right)_{\Omega_1} +
\rho_2\left(\frac{\partial u_{2}^{C}}{\partial t},v_{2}\right)_{\Omega_2} + \left(\beta_2 \nabla u_{2}^{C}, \nabla v_{2} \right)_{\Omega_2} \\
=
\left(f_{1},v_{1} \right)_{\Omega_1} + \left(f_{2},v_{2} \right)_{\Omega_2}, ~~~ \forall v_{}=(v_{1},v_{2})\in {\bm V}.
\end{split}
\end{equation}

\section{Semi-Discretization and Approximate Intrinsic or Inertial Type Robin Conditions}
The weak formulation is mathematically equivalent to the strong form of the original coupled PDE model in the conventional sense, and its solution still satisfies the original Dirichlet and Neumann conditions (2.3) and (2.4) on the interface.
There could be various equivalent interface conditions to be derived for various purposes. To understand what should be the appropriate interface conditions for a numerical scheme to satisfy for the sake of stability and convergence, in particular when we want to devise effective decoupling techniques later, let us start with the original continuous model and investigate the mechanism on how the interface conditions would change from the original model due to discrete approximations.

Note that the issue of numerical stability during time marching in numerical computation is caused by the discretization in both space and time.
To understand how the spacial discretizatin affects the interface conditions, yet without the stability factor, let us concentrate on the finite element semi-discrete approximation:

{\bf Model $S_h$: The Semi-Discrete Model with Finite Element Approximation in Space}\\

Find $u_{h}^{S}=(u_{1,h}^{S},u_{2,h}^{S})\in {\bm V}_h$, such that
\begin{equation}\label{ModelS}
\begin{split}
\rho_1\left(\frac{\partial u_{1,h}^{S}}{\partial t},v_{1,h}\right)_{\Omega_1} + \left(\beta_1 \nabla u_{1,h}^{S}, \nabla v_{1,h} \right)_{\Omega_1} +
\rho_2\left(\frac{\partial u_{2,h}^{S}}{\partial t},v_{2,h}\right)_{\Omega_2} + \\
\left(\beta_2 \nabla u_{2,h}^{S}, \nabla v_{2,h} \right)_{\Omega_2}
=
\left(f_{1,h},v_{1,h} \right)_{\Omega_1} + \left(f_{2,h},v_{2,h} \right)_{\Omega_2}, ~~\forall v_{h}=(v_{1,h},v_{2,h})\in {\bm V}_h,
\end{split}
\end{equation}
where ${\bm V}_h$ is a standard finite element subspace of ${\bm V}$.
Note now that {\bf Model $S_h$} is no longer equivalent to {\bf Model C}, but just an approximation. Yet there are standard error estimates for the solutions of these two models under certain regularity assumptions.

However, it is essentially important to point out that the approximate solution $u_{h}^{S}$ of {\bf Model $S_h$} no longer satisfies the interface conditions of the original coupled model although the Dirichlet condition remains enforced in the solution space ${\bm V}_h$, to which little attention has been paid in the past, mainly because the fully implicit time discretization will then lead to an unconditionally stable scheme, and the convergence and error analysis are ready in the classical finite element analysis, so it appears unnecessary to further investigate what happens on the interface.
However, if decoupling is necessary as a further approximation, it would then become crucially important to find out what is indeed the corresponding NEW interface coupling conditions implied by this approximate semi-discrete model due to the spacial discretization, which would then provide us insight for decoupling the numerical computation properly in order to keep maintaining numerical stability and convergence for decoupled methods.

Unfortunately, it is difficult to derive the local interface conditions for such a weak formulation defined in the integration form over the entire geometric domain.
For finding out such an interface condition approximately, motivated by \cite{fernandez2015generalized}, we propose to apply the lumped-mass technique \cite{thomee2006galerkin}, which will lead to a further approximation to the semi-discrete model.

Let us now briefly review the lumped-mass technique. Let $A$ be a $n\times n$ matrix and $A_L$ be the corresponding lumped-mass matrix
\begin{equation}
\begin{aligned}
&A
=\begin{bmatrix}

a_{11}  &  a_{12}  &            &          &           &    \\
a_{21}  &  a_{22}  &  a_{23}    &          &           &    \\
&  \ddots  & \ddots	    &  \ddots  & 	       &    \\	
&    	   &  a_{ii-1}  &  a_{ii}  &  a_{ii+1} &   \\
&    	   &            &  \ddots  &  \ddots   &\ddots   \\
&    	   &            &          &  \ddots   & \ddots    & \ddots      \\
&     	   &            &          &           &a_{nn-1}   & a_{nn} \\
\end{bmatrix}
\xrightarrow{\text{Lumped-mass}}\\
&A_L
=\begin{bmatrix}

a_{1}  &  0      &            &          &           &    \\
0      &  a_{2}  &  0    &          &           &    \\
&  \ddots  & \ddots	    &  \ddots  & 	       &    \\	
&    	   &  0  &  a_{i}  &  0 &   \\
&    	   &            &  \ddots  &  \ddots   &\ddots   \\
&    	   &            &          &  \ddots   & \ddots    & \ddots      \\
&     	   &            &          &           &0  & a_{n} \\
\end{bmatrix},
\end{aligned}
\end{equation}
where
\begin{equation}
a_i=\sum_{j=1}^{n} a_{ij}.
\end{equation}

Denote $W_h\subset H^{1}$ consists of continuous, piecewise linear functions on a quasiuniform family of triangulations $\mathcal{T}_h$ of any domain $\Omega$, where $\Omega=\Omega_1$ or $\Omega=\Omega_2$. The inner product in $W_h$ is defined by
$$\left(u_h,v_h\right)_{\Omega}=\int_{\Omega}u_h v_h, ~~~~\text{for any}~~ u_h,v_h\in W_h.$$

Let $\tau$ be a triangle of the triangulation $\mathcal{T}_h$ and $P_{\tau,j}$ be its vertices. Define the quadrature formula for any function $f$
\begin{equation}
Q_{\tau}(f) = \frac{1}{3}\text{area}{(\tau)}\sum_{j=1}^{3} f(P_{\tau,j}),
\end{equation}
then define the lumped-mass approximation of inner product $\left(u_h,v_h\right)_{\Omega}$ by
\begin{equation}\label{LMInnerProduct}
\left(u_h,v_h\right)_{\Omega,L} =  \sum_{\tau\in \mathcal{T}_h} Q_{\tau}\left(u_h v_h\right).
\end{equation}
Let $\{v_{i}\}_{i\leq n}\in W_h$ be the finite element basis, the following two properties hold \cite{thomee2006galerkin}:
\begin{itemize}
	\item $\left(v_i,v_j\right)_{\Omega,L}=0$ for $i\neq j$,
	\item $\left(v_i,v_i\right)_{\Omega,L}=\sum_{j=1}^{n}\left(v_i,v_j\right)_{\Omega}$.
\end{itemize}
From the properties, if $A$ is generated by $\left(v_i,v_j\right)_{\Omega}$, then $A_L$ is generated by $\left(v_i,v_j\right)_{\Omega,L}$.

Denote ${\bm V}_{1,h},{\bm V}_{1,h},{\bm \Lambda}_{\Gamma,h},{\bm V}_{1,h}^{\Gamma},{\bm V}_{2,h}^{\Gamma}$ be the corresponding discrete spaces with continuous piecewise linear functions. Applying the lumped-mass technique to ${\bm V}_{1,h},{\bm V}_{1,h}$, we may build connections among the inner products in ${\bm \Lambda}_{\Gamma,h}$ and ${\bm V}_{1,h},~{\bm V}_{2,h}$.
Define two discrete lifting operators
\begin{equation}\label{LiftOperators}
\begin{aligned}
\mathcal{L}_{1,h}:~{\bm \Lambda}_{\Gamma,h}\longrightarrow {\bm V}_{1,h},\\
\mathcal{L}_{2,h}:~{\bm \Lambda}_{\Gamma,h}\longrightarrow {\bm V}_{2,h},
\end{aligned}
\end{equation}
such that the nodal values of $\mathcal{L}_{1,h}\xi_h$ and $\mathcal{L}_{2,h}\xi_h$ vanish out of $\Gamma$ for all $\xi_h\in {\bm \Lambda}_{\Gamma,h}$. The interface operator ${\bm B}_{1,h}:~ {\bm \Lambda}_{\Gamma,h} \longrightarrow {\bm \Lambda}_{\Gamma,h}$ is defined by ${\bm B}_{1,h}=\mathcal{L}_{1,h}^{\ast}\mathcal{L}_{1,h}$, where $\mathcal{L}_{1,h}^{\ast}$ is the adjoint operator of $\mathcal{L}_{1,h}$ w.r.t the lumped-mass inner product in $V_{1,h}$. ${\bm B}_{2,h}$ is also defined similarly. Based on the discrete operators introduced above, we have
\begin{equation}\label{OperatorRelation1}
\begin{aligned}
&\left({\bm B}_{1,h} \xi_h, \lambda_h \right)_{\Gamma} =
\left(\mathcal{L}_{1,h}\xi_h, \mathcal{L}_{1,h}\lambda_h \right)_{\Omega_1,L}, ~~~\forall \xi_h,~\lambda_h\in {\bm \Lambda}_{\Gamma,h}\\
&\left({\bm B}_{2,h} \xi_h, \lambda_h \right)_{\Gamma} =
\left(\mathcal{L}_{2,h}\xi_h, \mathcal{L}_{2,h}\lambda_h \right)_{\Omega_2,L}, ~~~\forall \xi_h,~\lambda_h\in {\bm \Lambda}_{\Gamma,h}.
\end{aligned}
\end{equation}

The following Lemmas are useful in the analysis below. For simplicity, we denote $\mathcal{L}_{1}(v_{1}|_{\Gamma}),~\mathcal{L}_{2}(v_{2}|_{\Gamma})$ by $\mathcal{L}_{1}(v_{1}),~\mathcal{L}_{2}(v_{2})$.

\begin{lemma}\label{Lemma1}
	For all $\left( v_{1,h},\xi \right)\in {\bm V}_{1,h} \times {\bm \Lambda}_{\Gamma,h}$, we have
	\begin{equation}\label{OperatorRelation2_1}
	\left( v_{1,h}, \mathcal{L}_{1,h} \xi_h \right)_{\Omega_1,L} =
	\left({\bm B}_{1,h} v_{1,h}, \xi_h \right)_{\Gamma}.
	\end{equation}
\end{lemma}

\begin{proof}
	For $v_{1,h} \in {\bm V}_{1,h}$, it can be decomposed into $v_{1,h}=\tilde{v}_{1,h}+\mathcal{L}_{1,h} v_{1,h}$ with $\tilde{v}_{1,h} \in {\bm V}_{1,h}^{\Gamma}$. From the definition of lumped-mass inner product (\ref{LMInnerProduct}), we can conclude that
	$\left(\tilde{v}_{1,h}, \mathcal{L}_{1,h}\xi  \right)_{\Omega_1,L}=0$ for all $\xi_h \in {\bm \Lambda}_{\Gamma,h}$ since all the nodal values of $\tilde{v}_{1,h}\cdot\mathcal{L}_{1,h}\xi_h$ are $0$. Therefore, using (\ref{OperatorRelation1}), we have
	$$
	\left( v_{1,h}, \mathcal{L}_{1,h} \xi_h \right)_{\Omega_1,L}=
	\left( \tilde{v}_{1,h}, \mathcal{L}_{1,h} \xi_h \right)_{\Omega_1,L}+
	\left( \mathcal{L}_{1,h} v_{1,h}, \mathcal{L}_{1} \xi_h \right)_{\Omega_1,L}=
	\left( {\bm B}_{1,h} v_{1,h}, \xi_h \right)_{\Gamma}.
	$$
\end{proof}

Similarly, we also have
\begin{lemma}\label{Lemma2}
	For all $\left( v_{2,h},\xi_h \right)\in {\bm V}_{2,h} \times {\bm \Lambda}_{\Gamma,h}$, we have
	\begin{equation}\label{OperatorRelation2_2}
	\left( v_{2,h}, \mathcal{L}_{2,h} \xi_h \right)_{\Omega_2,L} =
	\left({\bm B}_{2,h} v_{2,h}, \xi_h \right)_{\Gamma}.
	\end{equation}
\end{lemma}

We may now define the lumped-mass semi-discrete model as a further approximation, denoted by {\bf Model $L_h$}.

{\bf Model $L_h$: The Lumped-Mass Semi-Discrete Model}\\

Find $u_{h}^{L}=(u_{1,h}^{L},u_{2,h}^{L})\in {\bm V}_h$, such that
\begin{equation}\label{ModelL}
\begin{split}
\rho_1\left(\frac{\partial u_{1,h}^{L}}{\partial t},v_{1,h}\right)_{\Omega_1,L} + \left(\beta_1 \nabla u_{1,h}^{L}, \nabla v_{1,h} \right)_{\Omega_1} +
\rho_2\left(\frac{\partial u_{2,h}^{L}}{\partial t},v_{2,h}\right)_{\Omega_2,L} + \\
\left(\beta_2 \nabla u_{2,h}^{L}, \nabla v_{2,h} \right)_{\Omega_2}
=
\left(f_{1,h},v_{1,h} \right)_{\Omega_1} + \left(f_{2,h},v_{2,h} \right)_{\Omega_2}, ~~\forall v_{h}=(v_{1,h},v_{2,h})\in {\bm V}_h.
\end{split}
\end{equation}

Under certain regularity assumptions, there are also standard error estimates for the solutions of the two semi-discrete models, as well as the original continuous model \cite{thomee2006galerkin}. In particular, if $u^C = \left( u_1^C, u_2^C \right)$ has $H^2$ regularity, for each subdomain we have

\begin{equation}\label{StandardErrorEstimates}
\begin{aligned}
&\|u_{i,h}^{L}-u_{i}^{C}\|_{L^{2}(\Omega_i)}=O(h^2),~~ i = 1, 2,\\
&\|u_{i,h}^{L}-u_{i}^{C}\|_{H^{1}(\Omega_i)}=O(h), ~~ i = 1, 2.\\
\end{aligned}
\end{equation}

There are many important applications in numerical computation where the Dirichlet condition enforced in the solution space ${\bm V}_h$ should be relaxed at certain steps for decoupling purposes. Such cases arise from matching the Dirichlet data on the interface by the classical non-overlapping domain decomposition approach for constructing parallel time marching schemes for parabolic problems or for constructing parallel iterative methods for elliptic problems by simulating the time evolution to the equilibrium state, to today's decoupling for multi-physics computation. For understanding the interface mechanism in order to construct effective decoupled numerical methods as well as to conduct the corresponding numerical analysis in the future, it will be fundamentally important to examine the following modified problem to the above {\bf Model $L_h$} by removing the Dirichlet interface condition from the solution space ${\bm V}_h$, where the modification of ${\bm V}_h$ is defined as ${\tilde{\bm V}}_h = \left\{ (v_{1,h},v_{2,h}) \in {\bm V}_{1,h} \times {\bm V}_{2,h}\right\}$.

{\bf Problem $P_h$: A Perturbation to Model $L_h$}\\

Consider a function $u_{h}^{\tilde{L}}=(u_{1,h}^{\tilde{L}},u_{2,h}^{\tilde{L}})\in {\tilde{\bm V}}_h$ satisfying
\begin{equation}\label{ModelP}
\begin{split}
\rho_1\left(\frac{\partial u_{1,h}^{\tilde{L}}}{\partial t},v_{1,h}\right)_{\Omega_1,L} + \left(\beta_1 \nabla u_{1,h}^{\tilde{L}}, \nabla v_{1,h} \right)_{\Omega_1} +
\rho_2\left(\frac{\partial u_{2,h}^{\tilde{L}}}{\partial t},v_{2,h}\right)_{\Omega_2,L} + \\
\left(\beta_2 \nabla u_{2,h}^{\tilde{L}}, \nabla v_{2,h} \right)_{\Omega_2}
=
\left(f_{1,h},v_{1,h} \right)_{\Omega_1} + \left(f_{2,h},v_{2,h} \right)_{\Omega_2}, ~ \forall v_{h}=(v_{1,h},v_{2,h})\in {{\bm V}}_h.
\end{split}
\end{equation}

The following theorem is revealing.

\begin{theorem}\label{theorem1}
If $u_{h}^{\tilde{L}} \in {\tilde{\bm V}}_h$ satisfies (\ref{ModelP}) in {\bf Problem $P_h$}, then $u_{h}^{\tilde{L}}$ satisfies the following condition:
\begin{equation}\label{WeakIntrinsicRobinConditionPerturbation}
\begin{split}
\rho_1\left({\bm B}_{1,h}\frac{\partial u_{2,h}^{\tilde{L}}}{\partial t},\xi_h  \right)_{\Gamma} +
\\
\left[\rho_2{\left(\frac{\partial u_{2,h}^{\tilde{L}}}{\partial t},\mathcal{L}_{2,h}\xi_h\right)_{\Omega_2,L}} +
{\left(\beta_2 \nabla u_{2,h}^{\tilde{L}}, \nabla(\mathcal{L}_{2,h}\xi_h) \right)_{\Omega_2}} -
{\left(f_{2,h},\mathcal{L}_{2,h}\xi_h \right)_{\Omega_2}} \right]
=\\
\rho_1\left({\bm B}_{1,h}\frac{\partial u_{1,h}^{\tilde{L}}}{\partial t},\xi_h  \right)_{\Gamma} -\\
\left[\rho_1\left(\frac{\partial u_{1,h}^{\tilde{L}}}{\partial t}, \mathcal{L}_{1,h}\xi_h  \right)_{\Omega_1,L}+
\left(\beta_1 \nabla u_{1,h}^{\tilde{L}}, \nabla(\mathcal{L}_{1,h}\xi_h) \right)_{\Omega_1} -
\left(f_{1,h},\mathcal{L}_{1,h}\xi_h \right)_{\Omega_1}\right] +\\
\rho_1\left({\bm B}_{1,h}\frac{\partial (u_{1,h}^{\tilde{L}}-u_{2,h}^{\tilde{L}})}{\partial t},\xi_h  \right)_{\Gamma}, ~~~ \forall \xi_h\in {\bm \Lambda}_{\Gamma,h}.
\end{split}
\end{equation}
\end{theorem}

\begin{proof}
Given $\xi_h\in {\bm \Lambda}_{\Gamma,h}$, taking $v_{1,h}=\mathcal{L}_{1,h}\xi_h$ and $v_{2,h}=\mathcal{L}_{2,h}\xi_h$ in (\ref{ModelP}), we have
\begin{equation}\label{ModelB_InterfacePart}
\begin{split}
&\left[\rho_1\left(\frac{\partial u_{1,h}^{\tilde{L}}}{\partial t},\mathcal{L}_{1,h}\xi_h\right)_{\Omega_1,L} +
\left(\beta_1 \nabla u_{1,h}^{\tilde{L}}, \nabla(\mathcal{L}_{1,h}\xi_h) \right)_{\Omega_1} -
\left(f_{1,h},\mathcal{L}_{1,h}\xi_h \right)_{\Omega_1}\right]
+\\
&\left[\rho_2{\left(\frac{\partial u_{2,h}^{\tilde{L}}}{\partial t},\mathcal{L}_{2,h}\xi_h\right)_{\Omega_2,L}} +
{\left(\beta_2 \nabla u_{2,h}^{\tilde{L}}, \nabla(\mathcal{L}_{2,h}\xi_h) \right)_{\Omega_2}}  -
{\left(f_{2,h},\mathcal{L}_{2,h}\xi_h \right)_{\Omega_2}}\right]
= 0.
\end{split}
\end{equation}
The lumped-mass property (\ref{OperatorRelation2_1}) in Lemma 3.1 allows us to condense the inner product in subdomain $\Omega_1$ to the interface $\Gamma$:
\begin{equation}\label{Temp8}
\rho_1\left(\frac{\partial u_{1,h}^{\tilde{L}}}{\partial t},\mathcal{L}_{{1,h}}\xi_h  \right)_{\Omega_1,L} =
\rho_1\left({\bm B}_{1,h}\frac{\partial u_{1,h}^{\tilde{L}}}{\partial t},\xi_h  \right)_{\Gamma}.
\end{equation}
In order to pass the Dirichlet information from $\Omega_1$ to  $\Omega_2$, denote $u_{1,h}^{\tilde{L}} = u_{2,h}^{\tilde{L}} - e_D$,
we have
\begin{equation}\label{Temp9}
\rho_1\left(\frac{\partial u_{1,h}^{\tilde{L}}}{\partial t},\mathcal{L}_{{1,h}}\xi_h  \right)_{\Omega_1,L} =
\rho_1\left({\bm B}_{1,h}\frac{\partial u_{2,h}^{\tilde{L}}}{\partial t},\xi_h  \right)_{\Gamma} -
\rho_1\left({\bm B}_{1,h}\frac{\partial e_D}{\partial t},\xi_h  \right)_{\Gamma}.
\end{equation}
Applying (\ref{Temp9}) to (\ref{ModelB_InterfacePart}) gives
\begin{equation}
\begin{split}
&\left[\rho_1\left({\bm B}_{1,h}\frac{\partial u_{2,h}^{\tilde{L}}}{\partial t},\xi_h  \right)_{\Gamma} +
\left(\beta_1 \nabla u_{1,h}^{\tilde{L}}, \nabla(\mathcal{L}_{1,h}\xi_h) \right)_{\Omega_1} -
\left(f_{1,h},\mathcal{L}_{1,h}\xi_h \right)_{\Omega_1}\right]
+\\
&\left[\rho_2{\left(\frac{\partial u_{2,h}^{\tilde{L}}}{\partial t},\mathcal{L}_{2,h}\xi_h\right)_{\Omega_2,L}} +
{\left(\beta_2 \nabla u_{2,h}^{\tilde{L}}, \nabla(\mathcal{L}_{2,h}\xi_h) \right)_{\Omega_2}}  -
{\left(f_{2,h},\mathcal{L}_{2,h}\xi_h \right)_{\Omega_2}}\right]
=\\
&\rho_1\left({\bm B}_{2,h}\frac{\partial e_D}{\partial t},\xi_h  \right)_{\Gamma}.
\end{split}
\end{equation}
Reorganizing the terms by separating the information on each subdomain yields
\begin{equation}\label{Temp6}
\begin{split}
\rho_1\left({\bm B}_{1,h}\frac{\partial u_{2,h}^{\tilde{L}}}{\partial t},\xi_h  \right)_{\Gamma} +\\
\left[\rho_2{\left(\frac{\partial u_{2,h}^{\tilde{L}}}{\partial t},\mathcal{L}_{2,h}\xi_h\right)_{\Omega_2,L}} +
{\left(\beta_2 \nabla u_{2,h}^{\tilde{L}}, \nabla(\mathcal{L}_{2,h}\xi_h) \right)_{\Omega_2}} -
{\left(f_{2,h},\mathcal{L}_{2,h}\xi_h \right)_{\Omega_2}} \right]
=\\ -
\left[ \left(\beta_1 \nabla u_{1,h}^{\tilde{L}}, \nabla(\mathcal{L}_{1,h}\xi_h) \right)_{\Omega_1} -
\left(f_{1,h},\mathcal{L}_{1,h}\xi_h \right)_{\Omega_1}\right]
+ \rho_1\left({\bm B}_{1,h}\frac{\partial e_D}{\partial t},\xi_h  \right)_{\Gamma}.
\end{split}
\end{equation}

Note that the left hand side in (\ref{Temp6}) consists of two part: the expression in the bracket corresponds to the terms in (\ref{ModelP}) from {\bf Problem $P_h$}  for subdomain {$\Omega_2$}, and the time derivative term being condensed to the interface from subdomain {$\Omega_1$}, but then exchanged the Dirichlet interface data with subdomain {$\Omega_2$} yet leaving an error term in the right hand side, which will play the key role in the intrinsic Robin condition to be derived.

For symmetry and later on to compare with the semi-discrete model, let us
add and subtract the same term $\rho_1\left(\frac{\partial u_{1,h}^{\tilde{L}}}{\partial t}, \mathcal{L}_{1,h}\xi_h  \right)_{\Omega_1,L}$ to equation (\ref{Temp6}), which gives
\begin{equation}
\begin{split}
\rho_1\left({\bm B}_{1,h}\frac{\partial u_{2,h}^{\tilde{L}}}{\partial t},\xi_h  \right)_{\Gamma} +\\
\left[\rho_2{\left(\frac{\partial u_{2,h}^{\tilde{L}}}{\partial t},\mathcal{L}_{2,h}\xi_h\right)_{\Omega_2,L}} +
{\left(\beta_2 \nabla u_{2,h}^{\tilde{L}}, \nabla(\mathcal{L}_{2,h}\xi_h) \right)_{\Omega_2}} -
{\left(f_{2,h},\mathcal{L}_{2,h}\xi_h \right)_{\Omega_2}} \right]
=\\
\rho_1\left(\frac{\partial u_{1,h}^{\tilde{L}}}{\partial t}, \mathcal{L}_{1,h}\xi_h  \right)_{\Omega_1,L} -\\
\left[\rho_1\left(\frac{\partial u_{1,h}^{\tilde{L}}}{\partial t}, \mathcal{L}_{1,h}\xi_h  \right)_{\Omega_1,L}+
\left(\beta_1 \nabla u_{1,h}^{\tilde{L}}, \nabla(\mathcal{L}_{1,h}\xi_h) \right)_{\Omega_1} -
\left(f_{1,h},\mathcal{L}_{1,h}\xi_h \right)_{\Omega_1}\right]\\
+ \rho_1\left({\bm B}_{2,h}\frac{\partial e_D}{\partial t},\xi_h  \right)_{\Gamma}.
\end{split}
\end{equation}

Using (\ref{Temp8}) again, the added inner product term outside of the bracket in the right hand side of the above equation may also be condensed to the interface, which yields
(\ref{WeakIntrinsicRobinConditionPerturbation}) and thus completes the proof.
\end{proof}

The fundamental significance of $Theorem$ \ref{theorem1} as well as its derivation is in that the error equation (\ref{WeakIntrinsicRobinConditionPerturbation}) reveals the interface mechanism of the semi-discrete model, in particular the impact of the error in the Dirichlet data on the interface conditions when the Dirichlet data is transmitted across the interface, which results instead in the error of the time derivative of the Dirichlet data in the form of $\rho_1 {\bm B}_{1,h} \frac{\partial e_D}{\partial t}$ on the interface. Mathematically, we note that the effect of the Dirichlet interface error $e_D$ at a given time $t$ here leads to an interface error of $\frac{\partial e_D}{\partial t}$ in the time evolution model, which implies how the Dirichlet type of information should be properly related on different time levels when time discretization is applied, in particular when decoupling is necessary along the interface. It explains why in most of the classical explicit type of decoupled schemes, the explicit use of Dirichlet data from previous time levels, such as the Dirichlet-Dirichlet, Dirichlet-Neumann type of decoupling methods, would usually lead to stability issues.
Most importantly, we note that when this scalar model problem is extended to other applications such as FSI, the Dirichlet variable will physically correspond to velocity,
while the related time derivative interface terms $\rho_1 {\bm B}_{1, h} \frac{\partial }{\partial t}$ in (\ref{WeakIntrinsicRobinConditionPerturbation}) will correspond to certain interface inertial force quantities due to spacial discretization, which will make the new interface condition to be derived physically meaningful in contrast to the classical Robin interface condition.

Notice that $e_D$ and thus $\frac{\partial e_D}{\partial t}$ become zero in the error equation (\ref{WeakIntrinsicRobinConditionPerturbation}) if the Dirichlet condition is enforced on the interface.
Therefore, as a special case of (\ref{WeakIntrinsicRobinConditionPerturbation}),
for (\ref{ModelL}) in {\bf Model $L_h$} we have the following interface relationship.

\begin{theorem}\label{lemma1}
If $u_{h}^{{L}}$ is a solution of {\bf Model $L_h$}, then $u_{h}^{{L}}$ satisfies the following interface condition:
	\begin{equation}\label{WeakIntrinsicRobinCondition}
	\begin{split}
	\rho_1\left({\bm B}_{1,h}\frac{\partial u_{2,h}^{L}}{\partial t},\xi_h  \right)_{\Gamma} +\\
	\left[\rho_2{\left(\frac{\partial u_{2,h}^{L}}{\partial t},\mathcal{L}_{2,h}\xi_h\right)_{\Omega_2,L}} +
	{\left(\beta_2 \nabla u_{2,h}^{L}, \nabla(\mathcal{L}_{2,h}\xi_h) \right)_{\Omega_2}} -
	{\left(f_{2,h},\mathcal{L}_{2,h}\xi_h \right)_{\Omega_2}} \right]
	=\\
	\rho_1\left({\bm B}_{1,h}\frac{\partial u_{1,h}^{L}}{\partial t},\xi_h  \right)_{\Gamma} -\\
	\left[\rho_1\left(\frac{\partial u_{1,h}^{L}}{\partial t}, \mathcal{L}_{1,h}\xi_h  \right)_{\Omega_1,L}+
	\left(\beta_1 \nabla u_{1,h}^{L}, \nabla(\mathcal{L}_{1,h}\xi_h) \right)_{\Omega_1} -
	\left(f_{1,h},\mathcal{L}_{1,h}\xi_h \right)_{\Omega_1}\right].
	\end{split}
	\end{equation}
\end{theorem}

Mathematically,
(\ref{WeakIntrinsicRobinCondition}) is equivalent to (\ref{ModelL}) in {\bf Model $L_h$} since the additional two terms $\rho_1\left({\bm B}_{1,h}\frac{\partial u_{2,h}^{L}}{\partial t},\xi_h  \right)_{\Gamma}$ and $\rho_1\left({\bm B}_{1,h}\frac{\partial u_{1,h}^{L}}{\partial t},\xi_h  \right)_{\Gamma}$ are canceled each other due to the Dirichlet interface condition being enforced in the solution space. However, they could make significant differences numerically when further approximation is applied for time discretization as well as decoupling.

To further understand the interface mechanism implied by (\ref{WeakIntrinsicRobinCondition}),
let us assume all the regularity requirements, well-posedness, as well as error estimates for the related continuous and semi-discrete models in order to focus on illustrating the derivation and understanding the mechanism of the intrinsic Robin condition to be derived.
Let us also first concentrate on the terms inside the two brackets in (\ref{WeakIntrinsicRobinCondition}), denoted by
\begin{equation}
\begin{split}
F=\\
&\left[
\rho_2{\left(\frac{\partial u_{2,h}^{L}}{\partial t},\mathcal{L}_{2,h}\xi_h\right)_{\Omega_2,L}} +
{\left(\beta_2 \nabla u_{2,h}^{L}, \nabla(\mathcal{L}_{2,h}\xi_h) \right)_{\Omega_2}} -
{\left(f_{2,h},\mathcal{L}_{2,h}\xi_h \right)_{\Omega_2}} \right]
+\\
&\left[
\rho_1\left(\frac{\partial u_{1,h}^{L}}{\partial t}, \mathcal{L}_{1,h}\xi_h  \right)_{\Omega_1, L} +
\left(\beta_1 \nabla u_{1,h}^{L}, \nabla(\mathcal{L}_{1,h}\xi_h) \right)_{\Omega_1} -
\left(f_{1,h},\mathcal{L}_{1,h}\xi_h \right)_{\Omega_1}\right],
\end{split}
\end{equation}
which actually corresponds to all the terms in (\ref{ModelL}) of the lumped-mass model.
Because (\ref{ModelL}) in {\bf Model $L_h$} is an approximation to (\ref{ModelS}) in {\bf Model $S_h$} due to the lumped-mass condensation of the $L^2$ inner product to the interface, we may relate the terms in $F$ back to (\ref{ModelS}) in {\bf Model $S_h$}  by approximating the lumped-mass terms using
$\rho_1\left(\frac{\partial u_{1,h}^{L}}{\partial t}, \mathcal{L}_{1,h}\xi_h  \right)_{\Omega_1}$ and $\rho_2{\left(\frac{\partial u_{2,h}^{L}}{\partial t},\mathcal{L}_{2,h}\xi_h\right)_{\Omega_2}}$ with an error of $O(h^2)$,
where the error between the inner product and its corresponding lumped-mass approximation in ${\bm V}_h$  is $O(h^2)$ as showed in \cite{thomee2006galerkin}. So, we have the estimate for $F$ with the solution of {\bf Model $L_h$}, yet in terms of the expression in {\bf Model $S_h$} as follows.

\begin{equation}\label{New2}
\begin{split}
F=\\
&\left[
\rho_2{\left(\frac{\partial u_{2,h}^{L}}{\partial t},\mathcal{L}_{2,h}\xi_h\right)_{\Omega_2}} +
{\left(\beta_2 \nabla u_{2,h}^{L}, \nabla(\mathcal{L}_{2,h}\xi_h) \right)_{\Omega_2}} -
{\left(f_{2,h},\mathcal{L}_{2,h}\xi_h \right)_{\Omega_2}} \right]
+\\
&\left[
\rho_1\left(\frac{\partial u_{1,h}^{L}}{\partial t}, \mathcal{L}_{1,h}\xi_h  \right)_{\Omega_1} +
\left(\beta_1 \nabla u_{1,h}^{L}, \nabla(\mathcal{L}_{1,h}\xi_h) \right)_{\Omega_1} -
\left(f_{1,h},\mathcal{L}_{1,h}\xi_h \right)_{\Omega_1}\right] + O(h^2).
\end{split}
\end{equation}

Now that {\bf Model $S_h$} is also the standard finite element approximation to {\bf Model C}, while {\bf Model $L_h$} is an approximation to {\bf Model C}, too.
Therefore, by using (\ref{StandardErrorEstimates}) we may approximate $u_{1,h}^{L},~u_{2,h}^{L}$ in (\ref{New2}) by the continuous solution $u_{1}^{C},~u_{2}^{C}$ with the error of $O(h)$ to further relate the expression $F$ to the weak form of the coupled model, yet restricted to the test function space ${\bm V}_h$:
\begin{equation}\label{New3}
\begin{split}
F=\\
&\left[
\rho_2{\left(\frac{\partial u_{2}^{C}}{\partial t},\mathcal{L}_{2,h}\xi_h\right)_{\Omega_2}} +
{\left(\beta_2 \nabla u_{2}^{C}, \nabla(\mathcal{L}_{2,h}\xi_h) \right)_{\Omega_2}} -
{\left(f_{2,h},\mathcal{L}_{2,h}\xi_h \right)_{\Omega_2}} \right]
+\\
&\left[
\rho_1\left(\frac{\partial u_{1}^{C}}{\partial t}, \mathcal{L}_{1,h}\xi_h  \right)_{\Omega_1} +
\left(\beta_1 \nabla u_{1}^{C}, \nabla(\mathcal{L}_{1,h}\xi_h) \right)_{\Omega_1} -
\left(f_{1,h},\mathcal{L}_{1,h}\xi_h \right)_{\Omega_1}\right] + O(h).
\end{split}
\end{equation}

Under sufficient regularity assumptions for the continuous solution $u^{C}$
and following the standard procedure, integrating by parts in (\ref{New3}) and using the strong form of the coupled PDE model leads to
\begin{equation}\label{New4}
F =
{{\left(\beta_2 \nabla u_{2}^{C} \cdot {\bm n}_2, \xi_h \right)_{\Gamma}}}
+
{{\left(\beta_1 \nabla u_{1}^{C} \cdot {\bm n}_1, \xi_h \right)_{\Gamma}}}
+O(h), ~~~ \forall \xi \in \Lambda_h.
\end{equation}
Therefore, we see that $F$ actually corresponds to the Newmann condition in the original interface matching conditions, or physically presenting the jump of the flux across the interface.

Combining (\ref{New4}) with (\ref{WeakIntrinsicRobinCondition}), we get
\begin{equation}\label{New6}
\begin{split}
{\rho_1\left({\bm B}_{1,h}\frac{\partial u_{2,h}^{L}}{\partial t},\xi_h  \right)_{\Gamma}} +
{{\left(\beta_2 \nabla u_{2}^{C} \cdot {\bm n}_2, \xi_h \right)_{\Gamma}}}
=\\
{\rho_1\left({\bm B}_{1,h}\frac{\partial u_{1,h}^{L}}{\partial t},\xi_h  \right)_{\Gamma}} -
{{\left(\beta_1 \nabla u_{1}^{C} \cdot {\bm n}_1, \xi_h \right)_{\Gamma}}}
+O(h).
\end{split}
\end{equation}

Approximating the continuous solution $(u_{1}^{C},~u_{2}^{C})$ in (\ref{New6}) back to the solution $(u_{1,h}^{L},~u_{2,h}^{L})$ of the approximate lumped-mass semi-discrete model, we find that (\ref{New6}) actually implies the following matching condition for $(u_{1,h}^{L},~u_{2,h}^{L})$ on the interface:
\begin{equation}\label{New7}
\begin{split}
&{\rho_1\left({\bm B}_{1,h}\frac{\partial u_{2,h}^{L}}{\partial t},\xi_h  \right)_{\Gamma}} +
{{\left(\beta_2 \nabla u_{2, h}^{L} \cdot {\bm n}_2, \xi_h \right)_{\Gamma^+}}} \\
=
&{\rho_1\left({\bm B}_{1,h}\frac{\partial u_{1,h}^{L}}{\partial t},\xi_h  \right)_{\Gamma}} -
{{\left(\beta_1 \nabla u_{1, h}^{L} \cdot {\bm n}_1, \xi_h \right)_{\Gamma^-}}}
+O(h),~~~ \forall \xi \in \Lambda_h,
\end{split}
\end{equation}
which may be formally written as
\begin{equation}\label{IntrinsicRobinCondition1WithError}
\rho_1{\bm B}_{1,h} \frac{\partial u_{2,h}^{L}}{\partial t} + \beta_2 \nabla u_{2,h}^{L} \cdot {\bm n}_2
= \rho_1{\bm B}_{1,h} \frac{\partial u_{1,h}^{L}}{\partial t} - \beta_1 \nabla u_{1,h}^{L} \cdot {\bm n}_1 + O(h),
~~~{\rm on~} \Gamma.
\end{equation}
Similar to the above argument, if we condense the inner product from $\Omega_2$ to the interface and transmit the Dirichlet data $u_{2,h}$ to the other side of the interface, we may have the following similar interface relation
\begin{equation}\label{IntrinsicRobinCondition2WithError}
\rho_2{\bm B}_{2,h} \frac{\partial u_{2,h}^{L}}{\partial t} + \beta_2 \nabla u_{2,h}^{L} \cdot {\bm n}_2
=
\rho_2{\bm B}_{2,h} \frac{\partial u_{1,h}^{L}}{\partial t} - \beta_1 \nabla u_{1,h}^{L} \cdot {\bm n}_1 + O(h),
~~~{\rm on~} \Gamma.
\end{equation}

First, note that in the above interface conditions (\ref{IntrinsicRobinCondition1WithError}) and (\ref{IntrinsicRobinCondition2WithError}), $\nabla u_{h}^{L}$ itself does not exist on $\Gamma$ because functions in the finite element space
${\bf V}_h$ have discontinuous derivatives on the interface, while the gradient of $u_{i,h}^{L}$ for $i = 1, 2$ is defined in the boundary layer cells in the subdomains on each side of the interface. Therefore, we denote each side of the interface in (\ref{New7}) by $\Gamma^+$ and $\Gamma^-$, respectively, to distinguish the interface values of a discontinuous function restricted from the corresponding interface side.

Most importantly, in contrast to the classical Robin interface condition, which simply takes a linear combination of the original Dirichlet and Neumann interface conditions, as introduced for constructing numerical methods for various decoupling purposes, this new type of interface condition instead combines the Neumann condition with the time derivative of the Dirichlet interface data, which physically describes the total exchange of heat or temperature balances on the interface and is referred to as the so-called intrinsic Robin condition.
When such a scalar model problem is further extended to other applications such as FSI, as pointed out earlier,
the related time derivative interface terms $\rho_1 {\bm B}_{1, h} \frac{\partial }{\partial t}$ in the new type of Robin condition will correspond to certain interface inertial force quantities according to Newton's second law, because the derivative of velocity gives acceleration, multiplying by the "mass" factor $\rho_1 {\bm B}_{1, h}$ then leads to a force.
We therefore may also call this new type of interface condition generally as
{\it the inertial type Robin condition}, because it is the inertial data then, instead of the velocity data, that are combined with the Neumann data from each side of the interface.

Indeed, together with the Dirichlet condition enforced in the solution space ${\bf V}_h$, the classical Robin condition, as well as the inertial type Robin condition (\ref{IntrinsicRobinCondition1WithError}) or (\ref{IntrinsicRobinCondition2WithError}) just introduced above, despite their approximation errors of $O(h)$, are all mathematically equivalent to the original interface condition set consisting of the Dirichlet and Neumann conditions.
However, the classical Robin condition is mathematically introduced by kind of artificially combining the Dirichlet and Neumann conditions to obtain another set of mathematically equivalent interface conditions, yet lack of physical justification.
Notice that, in applications like FSI, the Dirichlet data correspond to velocity, while the Neumann data correspond to interfacial force, therefore, there is a mismatch between these two different types of physical quantities. Such a physical mismatch is then kind of meredied by mathematically introducing a relaxation parameter, which is to be tuned for the purposes of numerical stability and convergence. The optimal relaxation parameter is then compensated for the mathematical effect of the physical mismatch. However, the optimal parameter analysis is often very difficult, even for certain simple model problems. The choice and tuning of the relaxation parameter with the existing numerical methods are practically very difficult and even impossible for most of the real applications, never mentioning the theoretical analysis and understanding of the numerical stability and convergence.

The inertial type of Robin condition, however, results from our error analysis in $Theorem$ \ref{theorem1} where we track how the information and its error are transmitted across the interface. The combination of the Dirichlet data, now through its time derivative instead, with the Neumann data becomes physically meaningful for applications like FSI because it characterizes how the inertial and stress forces should be balanced along the interface, which will in turn guide us in future approximation during time discretization and decoupling in order to maintain numerical stability and convergence properly.

\section{Full Discretization and Decoupling}\label{sec: Algorithm}
Implicit time discretization for the above semi-discrete models will lead to stable and convergent coupled methods. For illustration, let us consider the first order backward Euler scheme, although the discussions may be extended to high order schemes, such as the Crank-Nicolson Scheme. For simplicity, let us also consider linear finite elements in spacial discretization.

The standard finite element semi-discrete model is often applied to derive implicit coupled scheme by enforcing both Dirichlet and Neumann conditions on the current time level. To derive decoupled schemes, one may need to relax either Dirichlet or Neumann condition, or replace them by some other equivalent interface conditions such as the classical Robin condition or our new inertial type Robin condition. {\bf Problem $P_h$} illustrates how the Dirichlet data may be exchanged at certain step with the neighboring subdomain by using the computed data from the other side of the interface with the removal of the Dirichlet constraint from the solution space. For the purpose of relaxing the Neumann constraint at certain step instead of implicitly enforcing the Newmann condition as a natural interface condition imposed by {\bf Model C}, one may take the local PDEs (\ref{Equation1})-(\ref{Equation2}) in the original coupled model, multiplying by the corresponding test functions, integrating by parts, and then exchanging the Newmann information across the interface by applying the Newmann interface condition explicitly, and finally add them together to get a coupled form defined in the tensor product space of $V_1 \times V_2$. This will, of course, result in the corresponding interface terms involving the Neumann data on the other side of the interface, which may later be approximated by using computed information for the purpose of decoupling. For instance, we may formally derive the following quasi-weak formulation:

{\bf Problem EC: A Quasi-Weak Formulation of the Continuous Coupled Model}\\

Find $u_{}^{EC}=(u_{1}^{EC},u_{2}^{EC})\in {\bm V}$, such that
\begin{equation}\label{ModelEC}
\begin{split}
\rho_1\left(\frac{\partial u_{1}^{EC}}{\partial t},v_{1}\right)_{\Omega_1} +
\left(\beta_1 \nabla u_{1}^{EC}, \nabla v_{1} \right)_{\Omega_1} +
\rho_2\left(\frac{\partial u_{2}^{EC}}{\partial t},v_{2}\right)_{\Omega_2} + \\
\left(\beta_2 \nabla u_{2}^{EC}, \nabla v_{2} \right)_{\Omega_2}
= -
\left(\beta_2 \nabla u_{2}^{EC}\cdot {\bm n}_2,  v_{1} \right)_{\Gamma} -
\left(\beta_1 \nabla u_{1}^{EC}\cdot {\bm n}_1,  v_{2} \right)_{\Gamma} +\\
\left(f_{1},v_{1} \right)_{\Omega_1} + \left(f_{2},v_{2} \right)_{\Omega_2},~~ \forall v_{}=(v_{1},v_{2})\in {\bm \tilde{V}}.
\end{split}
\end{equation}

Note that we call this the quasi-weak formulation because the two interface terms
$\left(\beta_2 \nabla u_{2}^{EC}\cdot {\bm n}_2,  v_{1} \right)_{\Gamma}$ and
$\left(\beta_1 \nabla u_{1}^{EC}\cdot {\bm n}_1,  v_{2} \right)_{\Gamma}$ are not defined if $u^{EC}$ is only in the $H^1$ space ${\bm V}$, unless the PDE solution has further regularity of at least $H^{3/2}$ so that the traces of $\nabla u_{1}^{EC}$ and $\nabla u_{2}^{EC}$ may be defined on the interface.
However, the corresponding discrete counterparts are well defined as for the same reason as in the intrinsic or inertial Robin condition.
Its corresponding finite element approximation then reads:

{\bf Problem $ES_h$: A Semi-Discrete Model with Finite Element Approximation in Space}\\

Find $u_{h}^{ES}=(u_{1,h}^{ES},u_{2,h}^{ES})\in {\bm V}_h$, such that
\begin{equation}\label{ModelSP}
\begin{split}
\rho_1\left(\frac{\partial u_{1,h}^{ES}}{\partial t},v_{1,h}\right)_{\Omega_1} +
\left(\beta_1 \nabla u_{1,h}^{ES}, \nabla v_{1,h} \right)_{\Omega_1} +
\rho_2\left(\frac{\partial u_{2,h}^{ES}}{\partial t},v_{2,h}\right)_{\Omega_2} +\\
\left(\beta_2 \nabla u_{2,h}^{ES}, \nabla v_{2,h} \right)_{\Omega_2}
=
-\left(\beta_2 \nabla u_{2,h}^{ES}\cdot {\bm n}_2,  v_{1,h} \right)_{\Gamma^{-}} -
\left(\beta_1 \nabla u_{1,h}^{ES}\cdot {\bm n}_1,  v_{2,h} \right)_{\Gamma^{+}} +\\
\left(f_{1,h},v_{1,h} \right)_{\Omega_1} + \left(f_{2,h},v_{2,h} \right)_{\Omega_2}, ~~~\forall v_{h}=(v_{1,h},v_{2,h})\in \tilde{\bm V}_h.
\end{split}
\end{equation}

For comparison, let us first review a couple of typical decoupled methods by classical decoupling approaches in the literature.
Algorithm \ref{DNScheme} describes the decoupled Dirichlet-Neumann (DN) scheme \cite{lan2017mixed,quarteroni1999domain} based on decoupling
the implicit scheme from {\bf Problem $ES_h$} by applying the Dirichlet (\ref{InterfaceCondition1}) and Neumann interface conditions (\ref{InterfaceCondition2}) at alternating time levels with the use of computed data from the previous step on the other side of the interface during time marching.

\begin{algorithm}[tbhp]
	\caption{Decoupled DN Scheme.}
	\label{DNScheme}
	\begin{algorithmic}
		\State For $n=1,2,3...N$:
		\State 1. Solve local PDE in subdomain $\Omega_1$ with the Dirichlet condition (\ref{InterfaceCondition1}): Given $u_{2,h}^{n-1}$, find $u_{1,h}^n \in {\bm V}_{1,h}$, such that $u_{1,h}^n|_{\Gamma}=u_{2,h}^{n-1}|_{\Gamma}$ and
		$$	
		\begin{aligned}
		\left(\rho_1\frac{u_{1,h}^n - u_{1,h}^{n-1}}{\Delta t},v_{1,h}\right)_{\Omega_1} + \left(\beta_1 \nabla u_{1,h}^n, \nabla v_{1,h} \right)_{\Omega_1}=
		\left(f_{1,h}^n,v_{1,h} \right)_{\Omega_1} ~~~ \forall v_{1,h} \in {\bm V}_{1,h}.
		\label{WeakformDirichlet}
		\end{aligned}
		$$
		
		\State 2. Solve local PDE in subdomain $\Omega_2$ with the Neumann condition (\ref{InterfaceCondition2}): With $\nabla u_{1,h}^{n}$ computed above, find $u_{2,h}^n \in {\bm V}_{2,h}$, such that
		$$	
		\begin{aligned}
		&\left(\rho_2\frac{u_{2,h}^n - u_{2,h}^{n-1}}{\Delta t},v_{2,h}\right)_{\Omega_2} + \left(\beta_2 \nabla u_{2,h}^n, \nabla v_{2,h} \right)_{\Omega_2}
		=\\
		&\left(f_{2,h}^n,v_{2,h} \right)_{\Omega_2} -
		\left(\beta_1 \nabla u_{1,h}^n\cdot {\bm n}_1, v_{2,h} \right)_{\Gamma^{+}},
        ~~~\forall v_{2,h} \in {\bm V}_{2,h}.
		\label{WeakformNeumann}
		\end{aligned}
		$$
		
	\end{algorithmic}
\end{algorithm}

In the conventional Robin-Robin approach \cite{badia2008fluid,discacciati2007robin}, the equivalent classical Robin-Robin interface conditions are introduced by combining the original Dirichlet and Neumann conditions with two mathematical relaxation parameters $\alpha_1$ and $\alpha_2$ as
\begin{eqnarray}
\beta_1 \nabla u_1 \cdot {\bm n}_1 + \beta_2 \nabla u_2 \cdot {\bm n}_2 + \alpha_1 (u_1 - u_2 )=0\label{RobinCondition1},\\
\beta_1 \nabla u_1 \cdot {\bm n}_1 + \beta_2 \nabla u_2 \cdot {\bm n}_2 + \alpha_2 (u_2 - u_1 )=0\label{RobinCondition2}.
\end{eqnarray}

Algorithm \ref{RRScheme} describes the corresponding decoupled Robin-Robin (RR) scheme by applying the two Robin conditions above at alternating time levels during time marching. As discussed earlier, such an approach is lack of physical justification, and the determination of the optimal relaxation parameters is difficult both theoretically and practically for complicated real applications.

\begin{algorithm}[tbhp]
	\caption{Decoupled RR Scheme.}
	\label{RRScheme}
	\begin{algorithmic}
		\State For $n=1,2,3...N$:
		\State 1. Solve local PDE in subdomain $\Omega_1$ with Robin condition (\ref{RobinCondition1}): Given $u_{2,h}^{n-1}$ and $\nabla u_{2,h}^{n-1}$, find $u_{1,h}^n \in {\bm V}_{1,h}$, such that
		$$	
		\begin{aligned}
		\left(\rho_1\frac{u_{1,h}^n - u_{1,h}^{n-1}}{\Delta t},v_{1,h}\right)_{\Omega_1} + \left(\beta_1 \nabla u_{1,h}^n, \nabla v_{1,h} \right)_{\Omega_1} +
		\left(\alpha_1 u_{1,h}^n, v_{1,h} \right)_\Gamma = \\
		\left(f_{1,h}^n,v_{1,h} \right)_{\Omega_1} +
		\left(\alpha_1 u_{2,h}^{n-1}, v_{1,h} \right)_\Gamma - \left(\beta_2 \nabla u_{2,h}^{n-1}\cdot {\bm n}_{2,h}, v_{1,h} \right)_\Gamma,
		~~~\forall v_{1,h} \in {\bm V}_{1,h}.
		\label{WeakformRobin1}
		\end{aligned}
		$$
		
		\State 2. Solve local PDE in subdomain $\Omega_2$ with Robin condition (\ref{RobinCondition2}): With $u_{1,h}^{n}$ and $\nabla u_{1,h}^{n}$ computed above, find $u_{2,h}^n \in {\bm V}_{2,h}$, such that
		$$	
		\begin{aligned}
		\left(\rho_2\frac{u_{2,h}^n - u_{2,h}^{n-1}}{\Delta t},v_{2,h}\right)_{\Omega_2} + \left(\beta_2 \nabla u_{2,h}^n, \nabla v_{2,h} \right)_{\Omega_2} +
		\left(\alpha_2 u_{2,h}^n, v_{2,h} \right)_\Gamma = \\
		\left(f_{2,h}^n,v_{2,h} \right)_{\Omega_2} +
		\left(\alpha_2 u_{1,h}^{n}, v_{2,h} \right)_\Gamma - \left(\beta_1 \nabla u_{1,h}^n\cdot {\bm n}_{1,h}, v_{2,h} \right)_\Gamma,
		~~~\forall v_{2,h} \in {\bm V}_{2,h}.
		\label{WeakformRobin2}
		\end{aligned}
		$$

	\end{algorithmic}
\end{algorithm}

We now propose our new decoupling approach based on the intrinsic or inertial type Robin condition.
Notice that the decoupled {RR} Algorithm \ref{RRScheme} corresponds to adding the two Dirichlet terms $\alpha_1 (u_{1,h} - u_{2,h}, v_{1,h})_{\Gamma}$ and $\alpha_2 (u_{2,h} - u_{1,h}, v_{2,h})_{\Gamma}$ to {\bf Model $ES_h$} in the classical Robin interface condition to obtain another equivalent problem and then applying the decoupling technique similar to the Dirichlet-Neumann scheme. As discussed in the previous section, the lack of physical justification of the Robin-Robin approach results in both theoretical and practical numerical difficulties.
Mathematically, instead of adding the Dirichlet terms $\alpha_1 (u_{1,h} - u_{2,h}, v_{1,h})_{\Gamma}$ and $\alpha_2 (u_{2,h} - u_{1,h}, v_{2,h})_{\Gamma}$ , one may add any other term on both sides to {\bf Problem $ES_h$} to obtain another equivalent model without changing the solution, for instance the parameters $\alpha_1$ and $\alpha_2$ are generally arbitrary, just like in the classical fixed point iteration where there are infinite many possibilities one might try. The key is whether one may find a particular setting that leads to a better or even possibly optimal numerical algorithm.
Motivated by the interface data transmission analysis in the previous section, for instance an equivalent form (\ref{WeakIntrinsicRobinCondition}) for the lumped-mass finite element semi-discrete approximation, it is suggested that it indeed makes sense both mathematically and physically to instead add the time derivative terms
 $\rho_2\left({\bm B}_{2,h}(\frac{\partial u_{1,h}}{\partial t} - \frac{\partial u_{2,h}}{\partial t}),v_{1,h}  \right)_{\Gamma}$ or $\rho_1\left({\bm B}_{1,h}(\frac{\partial u_{2,h}}{\partial t} - \frac{\partial u_{1,h}}{\partial t}),v_{2,h}  \right)_{\Gamma}$.
This leads to an equivalent semi-discrete problem of {\bf Problem $ES_h$}:

{\bf Problem $ESI_h$: An Equivalent Semi-Discrete Problem based on the Intrinsic or Inertial Robin Conditions}\\
Find $u_{h}^{ESI}=(u_{1,h}^{ESI},u_{2,h}^{ESI})\in {\bm V}_h$, such that
\begin{equation}\label{ModelSI_h}
\begin{split}
\rho_1\left(\frac{\partial u_{1,h}^{ESI}}{\partial t},v_{1,h}\right)_{\Omega_1} + \left(\beta_1 \nabla u_{1,h}^{ESI}, \nabla v_{1,h} \right)_{\Omega_1} +
\rho_2\left(\frac{\partial u_{2,h}^{ESI}}{\partial t},v_{2,h}\right)_{\Omega_2} + \\
\left(\beta_2 \nabla u_{2,h}^{ESI}, \nabla v_{2,h} \right)_{\Omega_2}
+
\rho_2\left({\bm B}_{2,h}(\frac{\partial u_{1,h}^{ESI}}{\partial t} - \frac{\partial u_{2,h}^{ESI}}{\partial t}),v_{1,h}  \right)_{\Gamma}
=\\
-\left(\beta_2 \nabla u_{2,h}^{ESI}\cdot {\bm n}_2,  v_{1,h} \right)_{\Gamma^{-}} -
\left(\beta_1 \nabla u_{1,h}^{ESI}\cdot {\bm n}_1,  v_{2,h} \right)_{\Gamma^{+}} +\\
\left(f_{1,h},v_{1,h} \right)_{\Omega_1} + \left(f_{2,h},v_{2,h} \right)_{\Omega_2}, ~~~\forall v_{h}=(v_{1,h},v_{2,h})\in {\bm \tilde{V}}_h.
\end{split}
\end{equation}


Applying the backward Euler scheme gives an implicit and stable algorithm. However, this stable algorithm is coupled on the interface through the implicit approximation for the time derivative of the
two interface Dirichlet data terms $\frac{\partial u_{1,h}^n}{\partial t}$ and $\frac{\partial u_{2,h}^n}{\partial t}$, as well as the Dirichlet interface condition $u_{1,h}^n = u_{2,h}^n$ enforced
in the solution space.

Our framework described here will then allows one to apply various strategies to further decouple the implicit coupled scheme corresponding to {\bf Problem $ESI_h$}, while depending on the practical physical properties in more general real applications.
For instance, as motivated by the thin-wall FSI application \cite{fernandez2013explicit},
we may keep the implicit approximation for the time derivative of the interface Dirichlet data terms $\frac{\partial u_{1,h}^n}{\partial t}$ and $\frac{\partial u_{2,h}^n}{\partial t}$ on one side,
say $ \frac{\partial u_{1,h}^n}{\partial t}  \approx \frac{\delta u_{1,h}^n}{\delta t}  = \frac{u_{1,h}^n-{ u_{1,h}^{n-1}}}{\Delta t}$, while approximate the other term by using computed data from previous time levels explicitly on the other side of the interface,
say $ \frac{\partial u_{2,h}^n}{\partial t}  \approx \frac{\delta u_{2,h}^{n-1}}{\delta t}  = \frac{u_{2,h}^{n-1}-{ u_{2,h}^{n-2}}}{\Delta t}$.
Furthermore, to decouple the Dirichlet interface condition, we may remove the enforced condition $u_{1,h}^{n} = u_{2,h}^{n}$ in the solution space, while incorporate it into the algorithm by further approximating $u_{1,h}^{n-1}$ in the implicit approximation $ \frac{\delta u_{1,h}^n}{\delta t}  = \frac{u_{1,h}^n- u_{1,h}^{n-1}}{\Delta t}$ by passing the Dirichlet data from the other side with $u_{2,h}^{n-1}$, which implies to explicitly enforce the Dirichlet interface condition on the previous time level instead, namely,
$ \frac{\delta u_{1,h}^n}{\delta t} = \frac{u_{1,h}^n- u_{1,h}^{n-1}}{\Delta t} \approx \frac{u_{1,h}^n- u_{2,h}^{n-1}}{\Delta t}$.
This gives a decoupled intrinsic or inertial Robin-Neumann (iRN) type algorithm:

\begin{algorithm}[tbhp]
	\caption{Decoupled  iRN Scheme.}
	\label{IntrinsicRNSchemeB}
	\begin{algorithmic}
		\State For $n=1,2,3...N$:
		\State 1. Solve local PDE in subdomain $\Omega_1$ with the intrinsic or inertial Robin interface condition: \\
        Given $u_{2,h}^{n-1}$, $u_{2,h}^{n-2}$ and $\nabla u_{2,h}^{n-1}$, find $u_{1,h}^n \in {\bf V}_{1, h}$ such that
		$$	
		\begin{aligned}
		\left(\rho_1\frac{u_{1,h}^n - u_{1,h}^{n-1}}{\Delta t},v_{1,h}\right)_{\Omega_1} + \left(\beta_1 \nabla u_{1,h}^n, \nabla v_{1,h} \right)_{\Omega_1} +
		\left(\rho_2{\bm B}_{2,h} {\frac{\delta u_{1,h}^n}{\delta t}}  , v_{1,h} \right)_\Gamma =\\ \left(f_{1,h}^n,v_{1,h} \right)_{\Omega_1}
		+ \left({\rho_2\bm B}_{2,h} {\frac{\delta u_{2,h}^{n-1}}{\delta t}} , v_{1,h} \right)_\Gamma - \left(\beta_2 \nabla u_{2,h}^{n-1}\cdot {\bm n}_2, v_{1,h} \right)_{\Gamma^{-}},~~~
        \forall v_{1,h} \in {\bm V}_{1,h},
		\end{aligned}
		$$
		where $ {\frac{\delta u_{1,h}^n}{\delta t}}  = \frac{u_{1,h}^n- u_{1,h}^{n-1}}{\Delta t} \approx \frac{u_{1,h}^n- u_{2,h}^{n-1}}{\Delta t}$
		and
		$ {\frac{\delta u_{2,h}^{n-1}}{\delta t}}  = \frac{ u_{2,h}^{n-1} - u_{2,h}^{n-2} }{\Delta t}$;

		\State 2. Solve local PDE in subdomain $\Omega_2$ with the Neumann condition: \\
        With $\nabla u_{1,h}^{n}$ computed above, find $u_{2,h}^n \in {\bf V}_{2, h} $ such that
		$$	
		\begin{aligned}
		\left(\rho_2\frac{u_{2,h}^n - u_{2,h}^{n-1}}{\Delta t},v_{2,h}\right)_{\Omega_2} + \left(\beta_2 \nabla u_{2,h}^n, \nabla v_{2,h} \right)_{\Omega_2}
		= \\
		\left(f_{2,h}^n,v_{2,h} \right)_{\Omega_2}
		- \left(\beta_1 \nabla u_{1,h}^{n}\cdot {\bm n}_1, v_{2,h} \right)_{\Gamma^{+}},~~~ \forall v_{2,h} \in {\bm V}_{2,h}.
		\end{aligned}
		$$
		\\
		
	\end{algorithmic}
\end{algorithm}

Note that it is well known from the practical experiences that the numerical instability with decoupled methods is usually caused by making use of the computed Dirichlet data explicitly from the previous time level for decoupling the Dirichlet interface condition. It is the introduction of the time derivative of the Dirichlet data in the intrinsic or inertial Robin interface condition that is not only physically justified, but also mathematically allows for the implicit use of computed Dirichlet data from the previous time level to decouple the Dirichlet condition effectively.

In addition, we have only outlined above the basic principles in our decoupling framework based on the intrinsic or inertial type Robin interface conditions. The ingredients and strategies in a particular decoupled algorithm may vary and be fine tuned, such as whether to apply one inertial or Newmann interface term and how to decouple the Dirichlet and Newmann data by explicit approximation from using computed data, on which side or both sides or on alternating time levels, which would depend on the specific mathematical and physical properties in real applications, such as the ratio of the subdomain sizes, physical parameters like $\rho, \beta$, etc.
For instance, the strategy described in Algorithm \ref{IntrinsicRNSchemeB} works particularly effectively in the thin-wall FSI application, where it is observed that the implicit
interface approximation must be maintained on the bulk fluid side like $\Omega_1$ here, while the interface time derivation may be relaxed by applying an explicit approximation on
the thin-wall structure side for decoupling. Numerical experiments will be designed and conducted in the next section to confirm the effectiveness of the proposed decoupling approach, as well as to investigate some of these questions to provide certain insight and guidance for real applications.

\section{Numerical Experiments}

Numerical experiments are conducted to examine the effectiveness of the proposed decoupling approach based on the intrinsic or inertial type Robin interface conditions applied to the model problem in Section 2.
For illustration, let $\beta(x, y) = 2+x^2+y^2$, and the source term is defined such that the exact solution of the coupled model is given by
\begin{equation}
	u=t \sin(2\pi x)\sin(2\pi y).
\end{equation}
The total time is set to be $T=1s$ in all simulation.

We first examine the effectiveness of our new decoupled {iRN} scheme in Algorithm \ref{IntrinsicRNSchemeB} by comparing it with the corresponding implicit coupled scheme that is known to be numerically stable and convergent. For this purpose, we simply set $\rho_1=\rho_2=1$ and $\beta_1=\beta_2=\beta$.
The errors of the corresponding numerical solutions of the coupled and decoupled schemes to the exact solution are compared in Table \ref{IRN_rho_equal_beta_equal} with four levels of mesh refinement, where $\Delta t = h^2$ in order to examine the discretization error in space.
It is clearly observed that our new { iRN} scheme effectively decouples the computation by properly approximating the interface data using the computed solution, while still retains the same order of accuracy with the numerical solutions of the coupled scheme. The order of convergence is numerically verified because the errors are reduced by a factor of $4$ as the mesh size and time step are refined once, which indicates the second order accuracy in space and first order accuracy in time in the $L^2$ norm.
To further examine the stability of the decoupled iRN scheme, we increase $\Delta t$ from $=O(h^2)$ to $=O(h)$ and plot the computed solutions as well as the exact solution at $t=1$ in Figure \ref*{solution_stability_t=h_mesh}, where the stability and accuracy of the decoupled iRN scheme are clearly observed. The decoupled scheme is still numerically stable as the time step size keeps increasing, although the accuracy is reduced accordingly as expected due to bigger discretization error in time as well as in approximating the interface data from previous time levels for the decoupling purpose.

\begin{table}[tbhp]
	\small
	\caption{Errors of $\|u_{h,N}-u_{ext}(T)\|_{0,\Omega}$ with $\rho_1=\rho_2=1,~\beta_1=\beta_2=\beta(x, y),~ \Delta t= h^2.$}
	\label{IRN_rho_equal_beta_equal}
	\begin{center}			
		\begin{tabular}{|c|c|c|}
			\hline
			h									     			     &
			\quad Coupled scheme \quad               				 &
			Decoupled { iRN} scheme	\\
			\hline	
			$~~\frac{1}{8}~~$&
			3.43107e-2&
			3.75366e-2\\
			\hline
			$~~\frac{1}{16}~~$  &
			9.20988e-3&
			1.13574e-2\\
			\hline
			$~~\frac{1}{32}~~$&
			2.34374e-3&
			2.61407e-3\\
			\hline
			$~~\frac{1}{64}~~$&
			5.88545e-4&
			5.77793e-4\\
			\hline
		\end{tabular}
	\end{center}
\end{table}

\begin{figure}[tbhp]
	\centering
	\subfigure[Exact solution]{\includegraphics[scale=0.08]{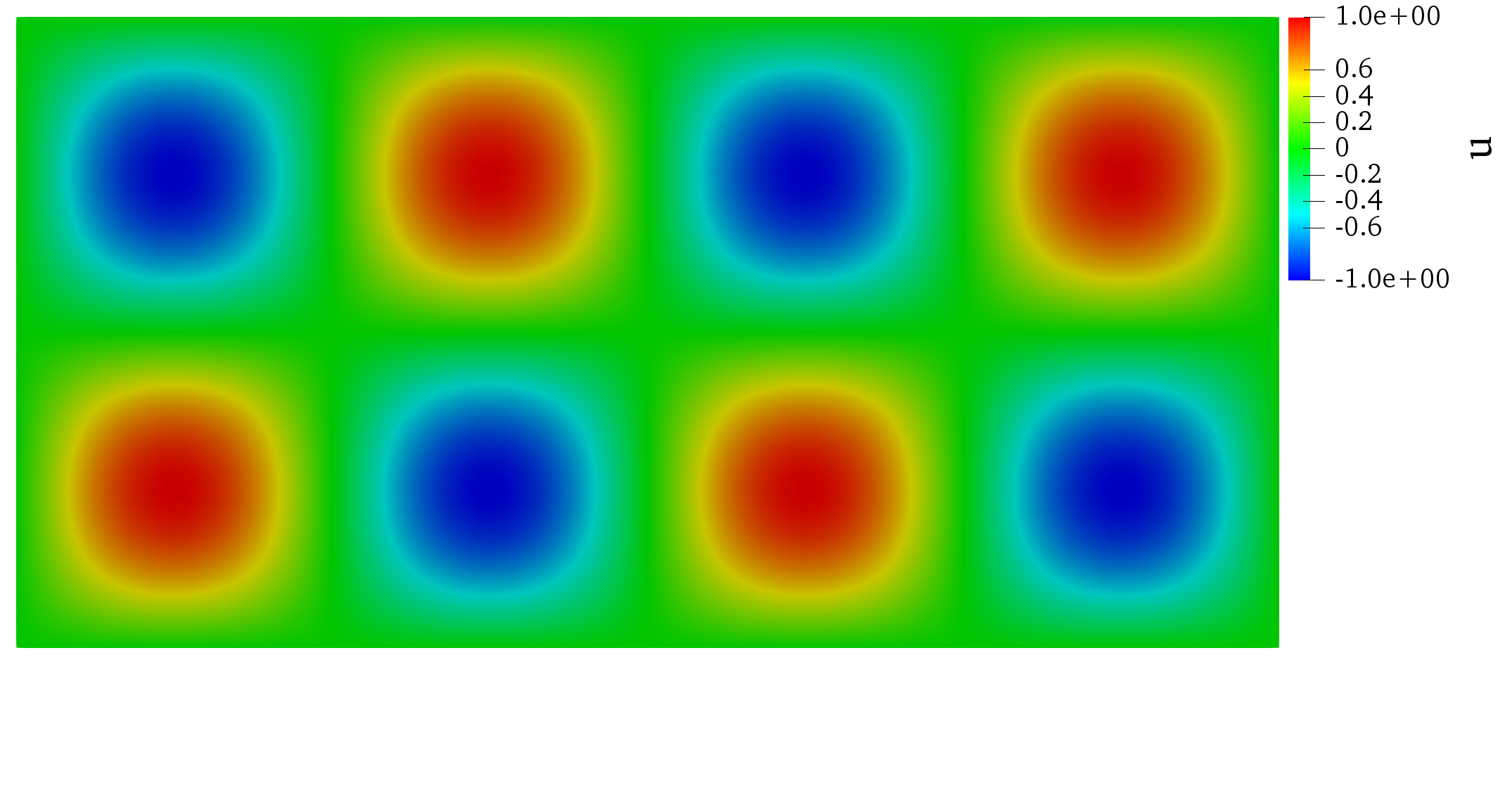}}
	\subfigure[$\Delta t=h=\frac{1}{8}$]{\includegraphics[scale=0.08]{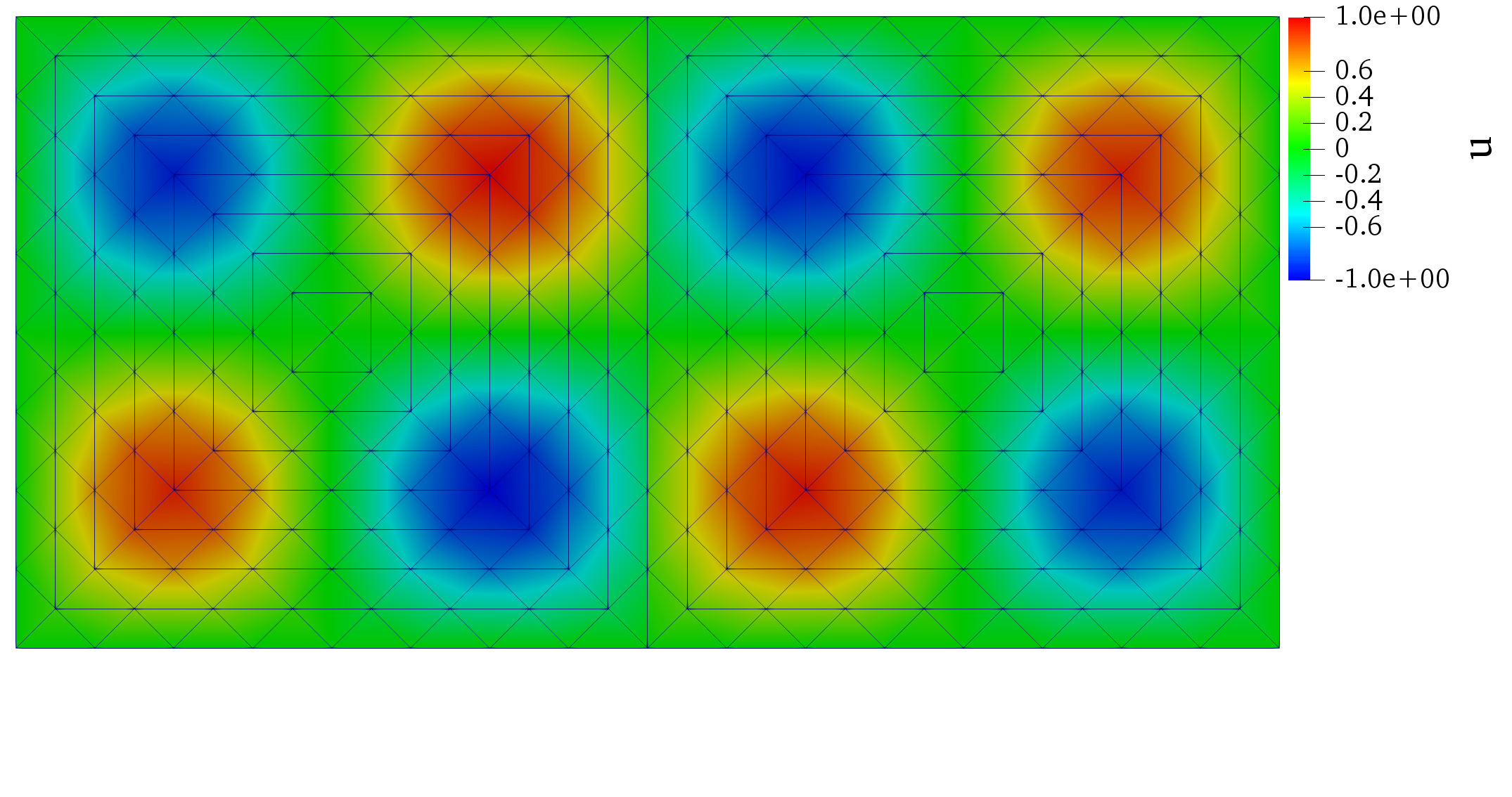}}\\
	\subfigure[$\Delta t=h=\frac{1}{16}$]{\includegraphics[scale=0.08]{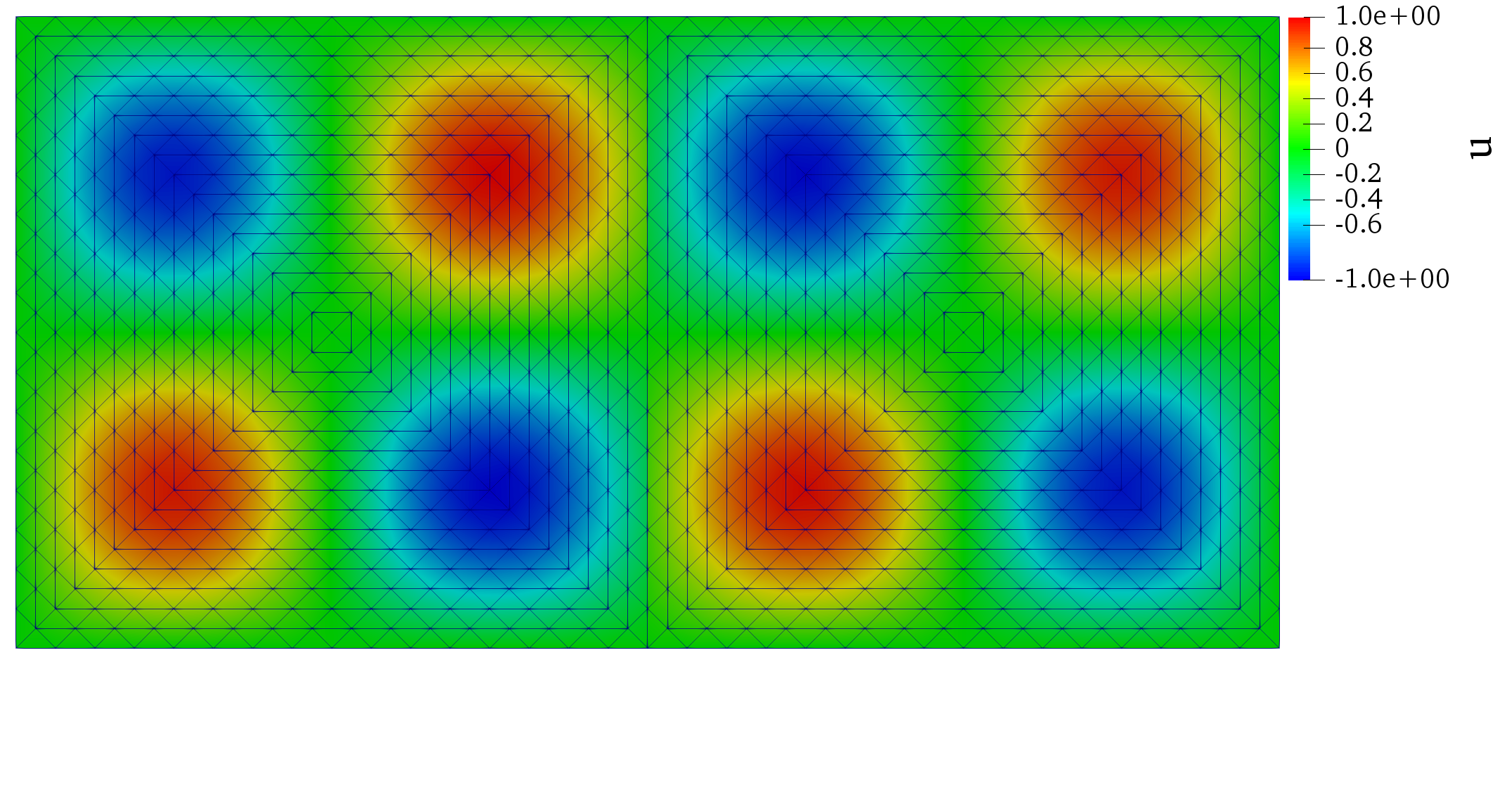}}
	\subfigure[$\Delta t=h=\frac{1}{32}$]{\includegraphics[scale=0.08]{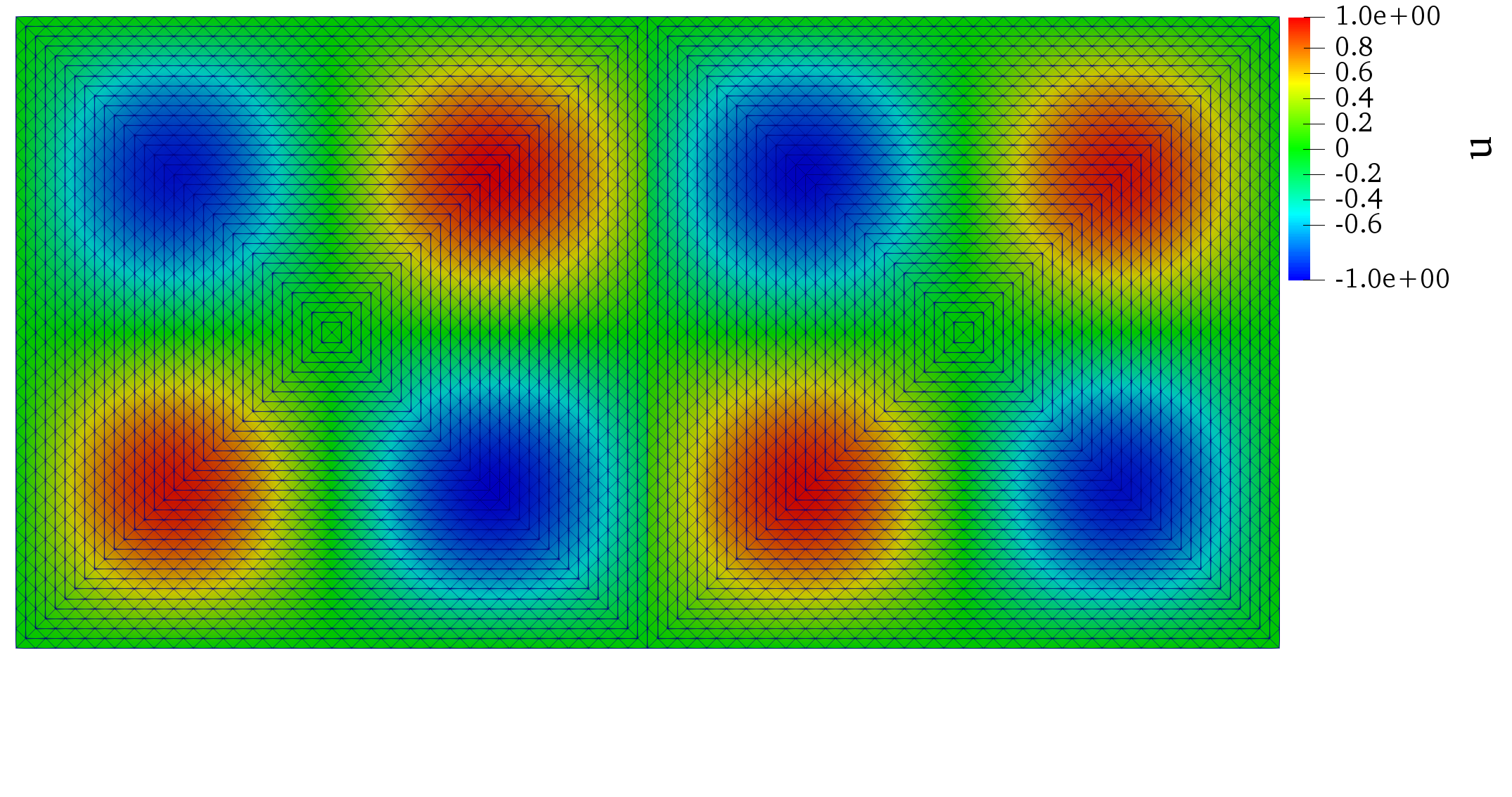}}
	\caption{A comparison of the exact solution and solutions obtained by the decoupled { iRN} scheme at $t=1$ with $\Delta t=h=\frac{1}{8},~\frac{1}{16},~\frac{1}{32}$.}
	\label{solution_stability_t=h_mesh}
\end{figure}

In order to compare with the classical decoupled {DN} and {RR} schemes, we may also consider to apply the corresponding intrinsic or inertial Robin conditions (\ref{IntrinsicRobinCondition1WithError}) and (\ref{IntrinsicRobinCondition2WithError}) symmetrically on both sides of the interface by adding another term $\rho_1\left({\bm B}_{1,h}(\frac{\partial u_{2,h}}{\partial t} - \frac{\partial u_{1,h}}{\partial t}),v_{2,h}\right)_{\Gamma}$ to {\bf Problem $ESI_h$}, which leads to another version of the intrinsic or inertial Robin-Robin{(iRR)} type of decoupled scheme:

\begin{algorithm}[tbhp]
	\caption{Decoupled {iRR} Scheme.}
	\label{IntrinsicRRSchemeB}
	\begin{algorithmic}
		\State For $n=1,2,3...N:$
		\State 1. Solve local PDE in subdomain $\Omega_1$ with the intrinsic or inertial Robin interface condition: \\
		Given $u_{2,h}^{n-1}$, $u_{2,h}^{n-2}$ and $\nabla u_{2,h}^{n-1}$, find $u_{1,h}^n \in {\bf V}_{1, h}$ such that
		$$	
		\begin{aligned}
		\left(\rho_1\frac{u_{1,h}^n - u_{1,h}^{n-1}}{\Delta t},v_{1,h}\right)_{\Omega_1} + \left(\beta_1 \nabla u_{1,h}^n, \nabla v_{1,h} \right)_{\Omega_1} +
		\left(\rho_2{\bm B}_{2,h} \frac{\delta u_{1,h}^n}{\delta t}, v_{1,h} \right)_\Gamma = \\
		\left(f_{1,h}^n,v_{1,h} \right)_{\Omega_1} +
		\left({\rho_2\bm B}_{2,h} \frac{\delta u_{2,h}^{n-1}}{\delta t}  , v_{1,h} \right)_\Gamma - \left(\beta_2 \nabla u_{2,h}^{n-1}\cdot {\bm n}_2, v_{1,h} \right)_{\Gamma^{-}},~~~
		\forall v_{1,h} \in {\bm V}_{1,h},
		\end{aligned}
		$$
		where $ {\frac{\delta u_{1,h}^n}{\delta t}}  = \frac{u_{1,h}^n- u_{1,h}^{n-1}}{\Delta t} \approx \frac{u_{1,h}^n- u_{2,h}^{n-1}}{\Delta t}$
		and
		$ {\frac{\delta u_{2,h}^{n-1}}{\delta t}}  = \frac{ u_{2,h}^{n-1} - u_{2,h}^{n-2} }{\Delta t}$;
		\State 2. Solve local PDE in subdomain $\Omega_2$ with the intrinsic or inertial Robin condition: \\
		With $u_{1,h}^{n}$, $u_{1,h}^{n-1}$ and $\nabla u_{1,h}^{n-1}$ computed above, find $u_{2,h}^n \in {\bf V}_{2, h} $ such that
		$$	
		\begin{aligned}
		\left(\rho_2\frac{u_{2,h}^n - u_{2,h}^{n-1}}{\Delta t},v_{2,h}\right)_{\Omega_2} + \left(\beta_2 \nabla u_{2,h}^n, \nabla v_{2,h} \right)_{\Omega_2} +
		\left(\rho_1{\bm B}_{1,h} \frac{\delta u_{2,h}^n}{\delta t}, v_{2,h} \right)_\Gamma = \\
		\left(f_{2,h}^n,v_{2,h} \right)_{\Omega_2} +
		\left(\rho_1{\bm B}_{1,h} \frac{\delta u_{1,h}^{n}}{\delta t}, v_{2,h} \right)_\Gamma - \left(\beta_1 \nabla u_{1,h}^n\cdot {\bm n}_1, v_{2,h} \right)_{\Gamma^{+}}, ~~~
		\forall v_{2,h} \in {\bm V}_{2,h},
		\end{aligned}
		$$
		where $ {\frac{\delta u_{2,h}^n}{\delta t}}  = \frac{u_{2,h}^n- u_{2,h}^{n-1}}{\Delta t} \approx \frac{u_{2,h}^n- u_{1,h}^{n-1}}{\Delta t}$
		and
		$ {\frac{\delta u_{1,h}^{n}}{\delta t}}  = \frac{ u_{1,h}^{n} - u_{1,h}^{n-1} }{\Delta t}$.
	\end{algorithmic}
\end{algorithm}

Similar to Table \ref{IRN_rho_equal_beta_equal},
we compare in Table \ref{table_sym_rho_equal_beta_equal_0} the errors between the exact solution and the numerical solutions of the decoupled { iRR} scheme with those of the classical decoupled {DN} and {RR} schemes, as well as the implicit coupled scheme as the reference.
In terms of effectiveness,
it is still observed that, for simple cases where if different decoupling techniques all work, such as our new { iRR} scheme, as well as the classical {DN} scheme and {RR} scheme with $\alpha_1=10,~\alpha_2=5$, they could have comparable effective performance. However, for the classical RR method, its performance relies on, and is very sensitive to the selection of relaxation parameters as discussed earlier, which would make it difficult for theoretical analysis, and most importantly for its application and extension to complicated real problems. For instance, it is shown in the table that the accuracy gets worse as the parameters are perturbated to $\alpha_1=\alpha_2=1$.
For the classical {DN} approach, although it works pretty well in this simple setting, it will be shown in the next experiment that it may fail, too when physical parameters vary.
In the contrary, our decoupled { iRR} scheme is always robust and parameter-free.
It is also numerically verified that the order of accuracy of the { iRR} scheme is similar to the { iRN} scheme.

\begin{table}[htpt!]
	\caption{Errors of $\|u_{h,N}-u_{ext}(T)\|_{0,\Omega}$ with $\rho_1=\rho_2=1,~\beta_1=\beta_2=\beta(x, y),~ \Delta t= h^2.$}
	\label{table_sym_rho_equal_beta_equal_0}
	\begin{center}		
		\begin{tabular}{|c|c|c|c|}
			\hline
			h									   &
			Coupled scheme~~~~         &
			\multicolumn{2}{|c|}{Decoupled DN scheme}\\
			\hline
			$~~\frac{1}{8}~~$&
			3.43107e-2       &
			\multicolumn{2}{|c|}{3.53891e-2}		 \\
			\hline
			$~~\frac{1}{16}~~$  &
			9.20988e-3          &
			\multicolumn{2}{|c|}{9.46749e-3}		    \\
			\hline
			$~~\frac{1}{32}~~$  &
			2.34374e-3      	&
			\multicolumn{2}{|c|}{2.40475e-3}			\\
			\hline
			\hline
			h									     			     &
			Decoupled { iRR} scheme	&
			\multicolumn{2}{|c|}{Decoupled RR scheme}				 \\
			\hline
			&
			&
			$~\alpha_1 = 10,\alpha_2 = 5~$ &
			$~\alpha_1 = 1,\alpha_2 = 1~$ \\
			\hline	
			$~~\frac{1}{8}~~$&
			3.53436e-2		 &
			3.24908e-2       &
			3.80028e-2       \\
			\hline
			$~~\frac{1}{16}~~$  &
			9.55398e-3		    &
			8.81404e-3          &
			1.39637e-2          \\
			\hline
			$~~\frac{1}{32}~~$  &
			2.25541e-3			&
			2.26911e-3          &
			5.08893e-3  		\\
			\hline
		\end{tabular}
	\end{center}
\end{table}

Let us finally illustrate the effects of physical properties by varying the value of
$\rho$ in different subdomains on different decoupling approaches of the three decoupled schemes as compared in Table \ref{table_sym_rho_equal_beta_equal_0}.
Two cases are tested, where
Table \ref{table_sym_rho_diff1_beta_equal_0} shows the results corresponding to the parameter set $\rho_1=10\rho_2$, while in Table \ref{table_sym_rho_diff2_beta_equal_0}  the parameter ratio is switched by $\rho_2=10\rho_1$.
It is clearly seen that the classical {DN} approach is very sensitive to physical properties and could fail if not applied carefully, while
the decoupled { iRR} scheme always works perfectly.

\begin{table}[H]
	\small
	\caption{Errors of $\|u_{h,N}-u_{ext}(T)\|_{0,\Omega}$ with $\rho_1=10,~\rho_2=1,~\beta_1=\beta_2=\beta(x, y),~ \Delta t= h^2$.}
	\label{table_sym_rho_diff1_beta_equal_0}
	\begin{center}		
		\begin{tabular}{|c|c|c|c|}
			\hline
			h									   &
			Coupled scheme~~~~         &
			\multicolumn{2}{|c|}{Decoupled DN scheme}\\
			\hline
			$~~\frac{1}{8}~~$&
			3.36012e-2       &
			\multicolumn{2}{|c|}{$\infty$}		 \\
			\hline
			$~~\frac{1}{16}~~$  &
			9.01136e-3         &
			\multicolumn{2}{|c|}{$\infty$}		    \\
			\hline
			$~~\frac{1}{32}~~$  &
			2.29268e-3      	&
			\multicolumn{2}{|c|}{$\infty$}			\\
			\hline
			\hline
			h									     			     &
			Decoupled { iRR} scheme	&
			\multicolumn{2}{|c|}{Decoupled RR scheme}				 \\
			\hline
			&
			&
			$~\alpha_1 = 10,\alpha_2 = 5~$ &
			$~\alpha_1 = 1,\alpha_2 = 1~$ \\
			\hline	
			$~~\frac{1}{8}~~$&
			3.24630e-2		 &
			3.18627e-2       &
			3.72954e-2       \\
			\hline
			$~~\frac{1}{16}~~$  &
			8.80706e-3		    &
			8.64307e-3          &
			1.37006e-2          \\
			\hline
			$~~\frac{1}{32}~~$  &
			2.27580e-3			&
			2.26911e-3          &
			5.00114e-3  		\\
			\hline
		\end{tabular}
	\end{center}
\end{table}

\begin{table}[H]
	\small
	\caption{Errors of $\|u_{h,N}-u_{ext}(T)\|_{0,\Omega}$ with $\rho_1=1,~\rho_2=10,~\beta_1=\beta_2=\beta(x, y),~ \Delta t= h^2.$}
	\label{table_sym_rho_diff2_beta_equal_0}
	\begin{center}		
		\begin{tabular}{|c|c|c|c|}
			\hline
			h									   &
			Coupled scheme~~~~         &
			\multicolumn{2}{|c|}{Decoupled DN scheme}\\
			\hline
			$~~\frac{1}{8}~~$&
			3.39101e-2       &
			\multicolumn{2}{|c|}{3.21780e-2}		 \\
			\hline
			$~~\frac{1}{16}~~$  &
			9.09826e-3         &
			\multicolumn{2}{|c|}{9.38034e-3}		    \\
			\hline
			$~~\frac{1}{32}~~$  &
			2.31506e-3      	&
			\multicolumn{2}{|c|}{2.38207e-3}			\\
			\hline
			\hline
			h									     			     &
			Decoupled { iRR} scheme	&
			\multicolumn{2}{|c|}{Decoupled RR scheme}				 \\
			\hline
			&
			&
			$~\alpha_1 = 10,\alpha_2 = 5~$ &
			$~\alpha_1 = 1,\alpha_2 = 1~$ \\
			\hline	
			$~~\frac{1}{8}~~$&
			3.23888e-2		 &
			3.21780e-2      &
			3.76394e-2       \\
			\hline
			$~~\frac{1}{16}~~$  &
			8.82536e-3		    &
			8.72945e-3         &
			1.38221e-2         \\
			\hline
			$~~\frac{1}{32}~~$  &
			2.28889e-3			&
			2.24812e-3          &
			5.04232e-3  		\\
			\hline
		\end{tabular}
	\end{center}
\end{table}

\section{Concluding Remarks}
We have derived an intrinsic or inertial type Robin condition for multi-modeling problems,
which is justified both mathematically and physically in contrast to the classical Robin condition.
Based on this new interface condition, a decoupling approach is presented for devising effective decoupled numerical methods for applications from classical parallel computing to multi-physics applications.
Numerical experiments show the effectiveness and robustness of our decoupling approach, its advantages over the existing decoupling approaches, as well as the promising potential for its application to complicated real problems in science and technology. Theoretical analysis for stability and convergence is under investigation.
	\section*{Acknowledgments}
	The authors' research is supported in part by the Hong Kong RGC Competitive Earmarked Research Grant HKUST16301218 and NSFC (91530319,11772281).
They would also like to thank Mingchao Cai and Yiyi Huang for valuable discussions.
	
	\newpage
	\bibliographystyle{siamplain}
	\bibliography{references}
\end{document}